\documentclass{temp}
\usepackage{amsfonts}
\usepackage{epstopdf}
\usepackage{bm}
\usepackage{amsmath}
\usepackage{hyperref}
\usepackage{multicol}
\usepackage{setspace}
\usepackage{graphicx}
\usepackage{caption}
\usepackage{subcaption}
\usepackage{float}
\usepackage{wrapfig}
\usepackage{lipsum}
\usepackage{array}

\DeclareMathOperator*{\esssup}{ess\,sup}
\DeclareMathOperator*{\essinf}{ess\,inf}
\renewcommand{\vec}[1]{{#1}}
\newcommand{\tens}[1]{{#1}}
\begin{document}

\title{MAXIMUM ENTROPY METHODS AS THE BRIDGE BETWEEN MICROSCOPIC AND MACROSCOPIC THEORY} 

\author{JAMIE M. TAYLOR}

\address{Mathematical Institute, University of Oxford, Andrew Wiles Building, Woodstock Road,\\ Oxford, OX2 6GG\\ jamie.taylor@maths.ox.ac.uk}

\bibliographystyle{acm}

\maketitle

\begin{history}
\received{(Day Month Year)}
\revised{(Day Month Year)}
\comby{(xxxxxxxxxx)}
\end{history}

\begin{abstract}
This paper investigates a function of macroscopic variables known as the singular potential, building on previous work by Ball and Majumdar. The singular potential is a function of the admissible statistical averages of probability distributions on a state space, defined so that it corresponds to the maximum possible entropy given known observed statistical averages, although non-classical entropy-like objective functions will also be considered. First the set of admissible moments must be established, and under the conditions presented in this work the set is open, bounded and convex allowing a description in terms of supporting hyperplanes, which provides estimates on the development of singularities for related probability distributions. Under appropriate conditions it is shown that the singular potential is strictly convex, as differentiable as the microscopic entropy and blows up uniformly as the macroscopic variable tends to the boundary of the set of admissible moments.  Applications of the singular potential are then discussed, and particular consideration will be given to certain free-energy functionals typical in mean-field theory, demonstrating an equivalence between certain microscopic and macroscopic free-energy functionals. This allows statements about $L^1$-local minimisers of Onsager's free energy to be obtained which cannot be given by two-sided variations, and overcomes the need to ensure local minimisers are bounded away from zero and infinity before taking bounded variations. The analysis also permits the definition of a dual order parameter for which Onsager's free energy allows an explicit representation. Also the difficulties in approximating the singular potential by everywhere defined functions, in particular by polynomials, are addressed with examples demonstrating the failure of the Taylor approximation to preserve shape properties of the singular potential.
\end{abstract}



\markboth{Jamie M. Taylor}{Maximum Entropy Methods As The Bridge Between Microscopic And Macroscopic Theory}

\section{Introduction}
In many-body problems in physics it is often required to reduce the complexity of the problem by applying statistical methods. Consider a state space $X$ with a corresponding measure $\mu$, and a probability distribution $\rho$ on $X$ that describes the probability of a given body occupying the state $t \in X$. Two particular examples are firstly nematic liquid crystals, where axially symmetric molecules can be described by their orientation, so that $X=\mathbb{S}^2$ \cite{mottram2004introduction}, and the Boltzmann equation with state space $X=\mathbb{R}^3\times \mathbb{R}^3$ corresponding to the position and momentum of particles \cite{harris2004introduction}. Thermodynamic equilibria can then be described as minima of free energy functionals on $\mathcal{P}(X)=\{\rho \in L^1(X):\int_X\rho\,d\mu=1,\rho\geq 0 \text{ a.e}\}$, the set of probability distributions on $X$. A particularly common example is the mean-field free energy, based on the second virial expansion and due to Onsager \cite{onsager1949effects} and is typically given by 
\begin{equation}
\mathcal{I}_{\mathcal{P}}(\rho)=T\int_X \rho(t) \ln \rho(t)\,d\mu(t) - \frac{1}{2} \int_X\int_X K(s,t)\rho(s)\rho(t)\,d\mu(t)\,d\mu(s).
\end{equation}
where the function $K \in L^\infty(X\times X)$ is a symmetric positive kernel, $T>0$ represents temperature and $\mu$ is a measure on $X$. This is analogous to the Helmholtz free energy, with the left-hand term representing an entropic contribution and the right-hand term representing chemical energy. This work will be concerned with the commonly considered case where the kernel $K$ is of the form 
\begin{equation}
K(s,t)=\sum\limits_{i,j=1}^k c_{ij}a_i(t)a_j(s)
\end{equation}
for some $k\in\mathbb{N}$, constants $c_{ij}$, and a set of linearly independent functions $a_i \in L^\infty(X)$ for $i=1,...,k$. The functions $a_i$ will often be denoted by a single vector valued function, $\vec{a}=(a_i)_{i=1}^k \in L^\infty(X,\mathbb{R}^k)$. An alternative way to approach such problems is by considering only macroscopic variables typically called order parameters, defined as statistical averages corresponding to the mean field model by
\begin{equation}
b_i = \int_X \rho(t)a_i(t)\,d\mu(t).
\end{equation}
The vector of moments will be denoted $\vec{b}=(b_i)_{i=1}^k \in \mathbb{R}^k$. By considering a finite dimensional state space of order parameters and disregarding the full statistical nature of the problem, analysis becomes much simpler through tools such as the Landau expansion (see for example Ref. \refcite{toledano1987landau} for a broad review). Within the context of the Q-tensor model for nematic liquid crystals, Majumdar noted that by losing the statistical nature of the problem, physical constraints on $b_i$ can be lost \cite{majumdar2010equilibrium}. For example from the H\"older inequality it must hold that $|b_i| \leq ||a_i||_\infty$ for each $i$. Motivated by the notation used in nematic liquid crystals, define $\mathcal{Q}$ to be the set of admissible order parameters,
\begin{equation}\mathcal{Q}=\left\{\left(\int_X a_i(t)\rho(t)\,dt\right)_{i=1}^k : \rho \in \mathcal{P}(X)\right\}\subset\mathbb{R}^k.\end{equation}
Majumdar investigated conditions under which equilibrium values of the order parameters are physical, in that they are elements of $\mathcal{Q}$, demonstrating that the Landau model can fail in this sense. Motivated by this problem Ball and Majumdar \cite{ball2010nematic} defined a {\it singular potential} $\psi_s$ on $\mathcal{Q}$ within the Q-tensor theory of nematics that builds on earlier work by Katriel {\it et. al.}\cite{katriel1986free} The singular potential is a convex function, inspired by the entropic term in the mean-field free energy, that blows up as the order parameters approach the boundary of $\mathcal{Q}$. Whilst the work of Ball and Majumdar concentrated on ensuring physicality in static problems relating to nematic liquid crystals, it has further been applied to dynamic systems \cite{feireisl2012evolution,wilkinson2012strict}, nematic elastomers \cite{calderer2013landau}, and the derivation of Q-tensor models \cite{han2013microscopic}, demonstrating the versatility of the framework. The aim of this work is to extend the ideas of Ball and Majumdar to define a singular potential in a more general setting, as well as analyse its properties and develop applications.

Before the singular potential can be defined the set of admissible moments $\mathcal{Q}$ must first be understood. This will form the bulk of Section \ref{sectionMomentProblem}. This is an example of the problem of moments \cite{akhiezer1965classical}, which in the general case is poorly understood. The classical problem considers $X=\mathbb{R}$, and functions $a_i(x)=x^i$, although others have considered the more general setting. For example, Lewis has provided a characterisation of the set of moments that roughly corresponds to the existence of solutions to certain optimisation problems \cite{lewis1995consistency}. In Section \ref{sectionMomentProblem} a more geometric description of the set $\mathcal{Q}$ will be useful, in which the set is described in terms of supporting hyperplanes. Whilst this description of $\mathcal{Q}$ is abstract and in general is unlikely to give an explicit expression for $\mathcal{Q}$ it is nonetheless appropriate for  the analysis in this work. In particular, this characterisation is used in providing growth bounds on the singular potential. Throughout both Sections \ref{sectionMomentProblem} and \ref{sectionSingularPotential} the assumption is made that the constraint functions $(a_i)_{i=1}^k$ and the constant function $a_0(t)= 1$ form a pseudo-Haar set of functions, a property first defined by Borwein and Lewis \cite{borwein1991duality}, which provides an elegant theory without being overly restrictive to applications in mean-field theory.

The singular potential is defined and analysed in Section \ref{sectionSingularPotential}. The singular potential corresponds to the greatest possible entropy of a probability distribution, subject to observed moments. This is an application of the principle of maximum entropy, pioneered in the seminal work of Jaynes \cite{jaynes1957information}, which has since been influential in a vast array of applications. Mathematically it can be phrased as a convex optimisation problem for the probability distribution subject to linear constraints. For the bulk of this work more general objective functions will be considered than the Shannon entropy, which is given by $\phi(x)=x\ln x$. The objective function will be required to share similar properties to the Shannon entropy (see Definition \ref{defEntropyLike}). One of these properties is that the objective function is not differentiable at $0$, which combined with the non-negativity constraint $\rho \geq 0$ has a consequence of not permitting two-sided variations about the minimiser $\rho^*$, unless it is known beforehand that $\essinf\limits_{t \in X} \rho^*(t)>0$. Borwein and Lewis \cite{borwein1991duality} provide an ideal framework for this kind of problem, based on duality, that will be exploited in this work. The main results of this section are that, under appropriate conditions, the singular potential $\psi_s$ is strictly convex, has the same differentiability as the objective function, and that $\psi_s(\vec{b})$ blows up to $+\infty$ uniformly as the distance from $\vec{b}$ to $\partial\mathcal{Q}$ approaches zero. 

Finally in Section \ref{sectionApplications} some potential applications are discussed. The results demonstrate an equivalence between certain minimisation problems in $\mathcal{P}(X)$ and $\mathcal{Q}$, and demonstrate the possibility of using the singular potential to rephrase harder questions relating to functional analysis in terms of simpler questions of several variable calculus. Particular mention is given to the mean-field approximation, where in the literature a rigorous treatment of the non-negativity constraint is often neglected, and the results in Section \ref{sectionApplications} provide an existence proof for minimisers as well as demonstrating that the solutions obtained using the first variation are correct. Furthermore the equivalence of local/global minimisers of the macroscopic and microscopic free energies has the implication that under the assumptions presented in this work, equilibria of mean-field free energies can be described using only macroscopic variables and the optimal entropy assumption. As noted by Decarreau {\it et. al} \cite{decarreau1992dual}, the dual formulation of the optimisation problem, which rephrases the optimisation problem in terms of finitely many Lagrange multipliers, is the most desirable from the point of numerical analysis. The ``moment-space" representation outlined in this work is loosely speaking equivalent, permits a simpler state space for analytical problems, and also allows the problem to be phrased directly in terms of order parameters. Expanding further on the work of Ball and Majumdar \cite{ball2014equilibrium}, models with spatial inhomogeneities will be considered and it is shown that for a particular class of free-energy functionals, minimisers are strictly physical, in the sense that the minimiser is bounded away from $\partial \mathcal{Q}$. However the class of models in which this analysis works is rather limited so that it remains open if more general free energies have strictly physical minimisers. 

The remainder of Section \ref{sectionApplications} discusses issues surrounding the approximation of the singular potential by globally defined functions, in particular by polynomials. The Landau theory of phase transitions \cite{toledano1987landau} states that since the free energy should be analytic in the order parameters, one can consider a Taylor expansion of the free energy to low (typically fourth) order and perform an energy minimisation over the approximation. It is shown that even within some simple one dimensional examples, the fourth order Taylor approximation to the singular potential in general fails to reproduce various desirable properties of the singular potential such as convexity, the existence of a single critical point and even a lower bound. Whilst it is possible to approximate the singular potential by convex polynomials due to a Weierstrass-type result, constructive methods appear to be out reach. Similarly, due to the difficulty of establishing the convexity of fourth order polynomials in several variables, even providing conditions for the Taylor approximation to be convex appears to be unobtainable. In response to this, the Yosida-Moreau approximation \cite{moreau1965proximite} is suggested as a globally defined approximation that preserves meaningful properties of the singular potential. The drawback however is that the Yosida-Moreau approximation is similarly defined through an optimisation procedure that, loosely speaking, is no less, although fortunately no more, difficult to solve than the optimisation problem that defines the singular potential itself. To illustrate issues with approximation and provide a concrete example of the singular potential framework, the McMillan model \cite{mcmillan1971simple} for isotropic-nematic-smectic A phase transitions is used as a guiding example. The isotropic phase corresponds to a disordered liquid, the nematic phase corresponds to molecules having orientational, but no positional order, and the smectic A phase corresponds to molecules having orientational order, as well as positional order in one direction, parallel to the molecular orientation. This model is described by a mean field free energy, with state space $X=\mathbb{S}^2\times[0,1]$, and two order parameters $(S,\sigma)$ defined as
\begin{equation}
\begin{split}
S =& \frac{1}{2}\int_0^1\int_{\mathbb{S}^2} \rho(\vec{p},x)\left(3\cos(\vec{p
}\cdot \vec{e}_1)^2-1\right)\,d\mathcal{H}^2(\vec{p})\,dx\\
\sigma =& \frac{1}{2}\int_0^1\int_{\mathbb{S}^2}\rho(\vec{p},x)\left(3\cos(\vec{p}\cdot \vec{e}_1)^2 -1\right)\cos(2\pi x) \, d\mathcal{H}^2(\vec{p})\,dx.
\end{split}
\end{equation}
Here $\vec{e}_1$ is a unit vector. If $S=\sigma=0$ then the system is in the isotropic phase, if $S\neq 0$ and $\sigma=0$ then the system is in the nematic phase, and finally if $S\neq 0 $ and $\sigma \neq 0$, the system is in the smectic A phase.

\section{The Moment Problem}
\label{sectionMomentProblem}

The key condition required for the framework presented to be successful is the pseudo-Haar property of a finite set of functions. 

\begin{definition}[Pseudo-Haar functions, see Ref. \refcite{borwein1992partially}]
Let $(X,\mu)$ be a finite measure space. Let $k \in \mathbb{N}$ and $f_i: X \to \mathbb{R}$, for $i=1,...,k$. The set $\{f_i:i=1,...,k\}$ is called pseudo-Haar if for every $Y \subset X$ with $\mu(Y)>0$, the set $\{f_i|_{Y}: i=1,...,k\}$ is linearly independent.
\end{definition}

\begin{remark}
If $(X,\mu)$ is a measure space with atoms, so that there exists some $t \in X$ with $\mu(\{t\})>0$, then any set of two or more functions on $X$ cannot be pseudo-Haar, since restricted to the set $\{t\}$ one is simply a multiple of the other. In particular the theory presented here is not applicable to discrete state spaces, such as $X$ a finite subset of $\mathbb{N}$. 
\end{remark}

Lewis \cite{lewis1995consistencyReport} showed that for $X$ a subset of $\mathbb{R}^n$, and $(f_i)_{i=1}^k$ analytic and linearly independent on a connected neighbourhood of $X$ then the functions are pseudo-Haar with respect to the Lebesgue measure. This result can be slightly extended to manifolds, which has applications in liquid crystal theory, where in particular $X$ is taken as $\mathbb{S}^2$ for axially symmetric molecules or $\text{SO}(3)$ for molecules with lower symmetry. By phrasing the problem in local coordinates, the result of Lewis can be used to extend the result to when the domain is a Riemannian manifold, and the measure taken will be the one induced by the metric, to be denoted by $\mu_g$. The result can trivially be extended for any other measure $\mu'$ such that for all $A \subset X$, $\mu'(A)=0\Rightarrow \mu_g(A)=0$. For a relatively self contained introduction to the theory of integration on manifolds, the reader is directed to Ref. \refcite{amann2009analysis}. Before attempting the proof, some preliminaries will be taken covered. 

\begin{definition}
Let be $X$ an analytic manifold with atlas $\left((U_\alpha,\phi_\alpha)\right)_{\alpha \in J}$, where $J$ is some index set. Say that $f:X\to\mathbb{R}$ is analytic if $f\circ \phi_\alpha^{-1}:\phi(U_\alpha)\to\mathbb{R}$ is analytic for all $\alpha \in J$.
\end{definition}

\begin{proposition}[Proposition 12.1.6 of Ref. \refcite{amann2009analysis}]
\label{propNullSetsManifold}
Let $(X,g)$ be an $n$-dimensional Riemannian manifold with induced measure $\mu_g$. Then $\mu_g(A)=0$ if and only if $\mathcal{L}^n(\phi(A\cap U))=0$ for all charts $(U,\phi)$ of $X$. 
\end{proposition}

\begin{proposition}[From Ref. \refcite{lewis1995consistencyReport}]
\label{propLewisPH}
Let $\Omega \subset \mathbb{R}^n$ be open and connected, and let $f:\Omega\to\mathbb{R}$ be analytic. Then if there exists some set $A \subset \Omega$ with $\mathcal{L}^n(A)>0$ so that $f|_A=0$, then $f=0$. 
\end{proposition}

Using these results it is now straightforward to obtain the required result.

\begin{proposition}
\label{propPseudoHaarManifold}
Let $(X,g)$ be a connected, analytic, Riemannian manifold. Denote by $\mu$ the measure on $X$ induced by the metric $g$. Let $f_i : X \to \mathbb{R}$ be analytic functions on $X$ for $i=1,..,k$ for some $k \in \mathbb{N}$. Then the set $\{f_i : i =1,..,k\}$ is linearly independent if and only if it is pseudo-Haar. 
\end{proposition}

\begin{proof}
If the functions are pseudo-Haar then linear independence is immediate. To show the converse, assume the result is false. This implies that there exists some $\vec{\xi} \in \mathbb{R}^k$ with $\vec{\xi} \neq \vec{0}$ and $A \subset X$ so that $\mu_g(A)>0$ and $\sum\limits_{i=1}^k \xi_if_i(x)=0$ for all $x \in A$. For simplicity denote $F=\sum\limits_{i=1}^k\xi_i f_i$. Then it is clear that $F$ is analytic, and $F|_A=0$. Without loss of generality take all charts to be connected by splitting the charts into connected components where necessary. Using Prop. \ref{propNullSetsManifold}, there must exist some chart $(U,\phi)$ so that $\mathcal{L}^n(\phi(A\cap U))>0$. For brevity let $\tilde{U}=\phi(U)$ and $\tilde{A}=\phi(A\cap U)$. Also let $\tilde{F}=F\circ \phi^{-1}:\tilde{U}\to\mathbb{R}$. Since $F$ is analytic, then $\tilde{F}$ is analytic. Therefore since $\tilde{F}|_{\tilde{A}}=0$ with $\mu_g(\tilde{A})>0$, and $\tilde{F}$ analytic on the connected open set $\tilde{U}$, by Prop. \ref{propLewisPH} it must hold that $\tilde{F}=0$ on $\tilde{U}$. Composing $\tilde{F}$ with $\phi$ gives that $F=0$ on $U$, which is open. Therefore since $X$ is connected and analytic, $F$ is zero on all of $X$. Recalling the definition of $F$, this implies $\sum\limits_{j=1}^k \xi_if_i=0$, contradicting the linear independence assumption and completing the proof. \qed
\end{proof}

\begin{definition}[The sets $\mathcal{P}(X)$ and $\mathcal{Q}$]
Let $(X,\mu)$ be a finite measure space. Consider finitely many constraint functions $a_i \in L^\infty(X)$ for $i=1,...,k$, such that the set $\{t\mapsto 1\}\cup\{a_i:i=1,..,k\}$ is pseudo-Haar. For ease of notation let $\vec{a}=(a_i)_{i=1}^k \in L^\infty(X,\mathbb{R}^k)$. Define the sets
\begin{equation}
\begin{split}
\mathcal{P}(X) &= \left\{ \rho \in L^1(X) : \rho\geq 0 \,\,\,\mu\text{-a.e.}, \, \int_X \rho\, d\mu = 1\right\},\\
\mathcal{Q}&= \left\{ \int_X \vec{a} \rho \, d\mu : \rho \in \mathcal{P}(X) \right\} \subset \mathbb{R}^k.
\end{split}
\end{equation}
Furthermore, if $\vec{b} \in \mathcal{Q}$, then $\vec{b}$ will be called physical, and if $\vec{b}=\int_X \rho \vec{a}\,d\mu$, then $\vec{b}$ is generated by $\rho$. 
\end{definition}

\begin{definition}
\label{defES}
Let $\vec{u} \in \mathbb{S}^{k-1}$. Given the constraint functions $(a_i)_{i=1}^k$ with $a_i \in L^\infty(X)$ for $i=1,..,k$, define $S_{\vec{u}} \in \mathbb{R}$ by
\begin{equation}
S_{\vec{u}} = \esssup\limits_{t \in X} \vec{u} \cdot \vec{a}(t).
\end{equation}
Let $\epsilon > 0$. Define the set $E_\epsilon^{\vec{u}}\subset X$ as 
\begin{equation}
E_\epsilon^{\vec{u}} = \left\{ t \in X : S_{\vec{u}} < {\vec{u}}\cdot \vec{a}(t)+\epsilon\right\}.
\end{equation}
\end{definition}

Note that since $\mu(X)<\infty$, $L^\infty(X) \subset L^2(X)$, so that up to an invertible linear transformation, it can be assumed that the functions  $(a_i)_{i=1}^k$ are $L^2(X)$ orthonormal functions, so that $\int_X a_i a_j\,d\mu =\delta_{ij}$, and orthogonal to any constant so that $\int_X a_i \, d\mu=0$. For simplicity of calculation, this will be assumed unless stated otherwise, although it has little effect on the results since the non-normalised framework is equivalent up to an affine map on $\mathcal{Q}$. This has the consequence that the uniform distribution, $\rho_U(t)=\frac{1}{\mu(X)}$, satisfies $\int_X \rho_U(t) \vec{a}(t) \, d\mu(t)=\vec{0}$. This is relevant in Landau theory where the order parameters are required to be zero in the high temperature uniform state. In order to establish a necessary and sufficient condition for $\vec{b} \in \mathcal{Q}$, some properties of $\mathcal{Q}$ will be established. The method is to describe $\mathcal{Q}$ in terms of supporting hyperplanes.

\begin{proposition}
The set $\mathcal{Q}$ is convex, bounded and open.
\end{proposition} 
\begin{proof}
Convexity is immediate from the convexity of $\mathcal{P}(X)$. 
To show boundedness, let $\vec{b} \in \mathcal{Q}$, corresponding to some $\rho \in \mathcal{P}(X)$. Then
\begin{equation}
\begin{split}
|\vec{b}|&=\left|\int_X \vec{a}(t) \rho(t) \, d\mu(t)\right|\\ 
&\leq  \esssup\limits_{t\in X} |\vec{a}(t)|
\end{split}
\end{equation}
so that $\mathcal{Q}$ is bounded. To see that $\mathcal{Q}$ is open, consider some $\vec{b} \in \mathcal{Q}$ and let $\vec{\xi} \in \mathbb{R}^k$. Since the functions $t \mapsto 1$ and $a_i$ are pseudo-Haar, by Ref. \refcite[Theorem 2.9]{borwein1991duality} there exists some $\rho \in \mathcal{P}(X)$ that generates $\vec{b}$ and $\epsilon > 0$ such that $\rho(t) \geq \epsilon > 0$ almost everywhere. Define the function $\chi_{\vec{\xi}}\in L^\infty(X)$ by 
\begin{equation}
\chi_{\vec{\xi}}(t)= \sum\limits_{i=1}^k a_i(t) \xi_i.
\end{equation}
Then $\int_X \chi_{\vec{\xi}} \, d\mu = 0$ and $\int_X \vec{a}(t) \chi_{\vec{\xi}}(t) \, d\mu(t) = \vec{\xi}$ due to the orthogonality of $(a_i)_{i=1}^k$ and the constant function. Note that $|\chi_{\vec{\xi}}(t)|= |\vec{\xi} \cdot a(t) | \leq |\vec{\xi}| \, \textup{ess}\sup\limits_{s \in X} |\vec{a}(s)|$. Hence if 
\begin{equation}
|\vec{\xi}|< \frac{\epsilon}{\esssup\limits_{s \in X} |\vec{a}(s)|},
\end{equation}
then $\rho_{\vec{\xi}}=\rho + \chi_{\vec{\xi}}$ is non-negative almost everywhere. Combining the previous results, this gives that $\rho_{\vec{\xi}} \in \mathcal{P}(X)$, and $\int_X \vec{a}(t) \rho_{\vec{\xi}}(t) \, d\mu(t) = \vec{b} + \vec{\xi}$, so that $\mathcal{Q}$ is open. 
\end{proof}

Using these properties it is now possible to characterise $\mathcal{Q}$ in terms of its supporting hyperplanes, and this will provide the necessary and sufficient condition for $\vec{b\in}\mathcal{Q}$.

\begin{theorem}
\label{theoremNecSufCond}
Let $\vec{b} \in \mathbb{R}^k$. Then $\vec{b} \in \mathcal{Q}$ if and only if 
\begin{equation}
\vec{b} \cdot \vec{u} < \esssup\limits_{t \in X} \vec{a}(t)\cdot \vec{u}=S_{\vec{u}}
\end{equation}
for all $\vec{u} \in \mathbb{S}^{k-1}$. 
\end{theorem}
\begin{proof}
By Ref. \refcite[Theorem 11.5]{rockafellar1997convex} it is known that a closed convex set $K\subset \mathbb{R}^k$ can be written as the intersection of all closed half spaces containing $K$, so that $\vec{x} \in K$ is equivalent to 
\begin{equation}
\vec{x}\cdot \vec{u} \leq \sup\limits_{\vec{z} \in K} \vec{z} \cdot \vec{u}
\end{equation}
for all $\vec{u} \in \mathbb{S}^{k-1}$. If $\Omega \subset \mathbb{R}^k$ is a bounded open convex set, then by considering $K=\bar{\Omega}$ it is straightforward to deduce that $\vec{x} \in \Omega$ is equivalent to 
\begin{equation}
\vec{x} \cdot \vec{u} < \sup\limits_{\vec{z} \in K} \vec{z} \cdot \vec{u} = \sup\limits_{\vec{z} \in \Omega}\vec{z}\cdot \vec{u}
\end{equation}
for all $\vec{u} \in \mathbb{S}^{k-1}$, since a non-zero linear function cannot attain its maximum over a set in the interior. By taking $\Omega = \mathcal{Q}$, the supremum can be computed as
\begin{equation}
\begin{split}
\sup\limits_{\vec{z} \in \mathcal{Q}} \vec{z} \cdot \vec{u} &= \sup\limits_{\rho \in \mathcal{P}(X)} \int_X \vec{u} \cdot \vec{a}(t) \rho(t) \, d\mu(t)\\
&= \esssup\limits_{t \in X} \vec{u} \cdot \vec{a}(t)\\
&=S_{\vec{u}},
\end{split}
\end{equation}
by considering $\rho$ uniform on the set $E_\epsilon^{\vec{u}}=\{t \in X : \text{ess}\sup \vec{u}\cdot \vec{a}(s)<\vec{u}\cdot \vec{a}(t)+\epsilon\}$ and zero elsewhere, then taking $\epsilon \to 0$, which gives the necessary and sufficient condition.\qed
\end{proof}

Using the necessary and sufficient condition given previously it is possible to formalise the intuition that if $\vec{b}$ is close to the boundary of $\mathcal{Q}$, any probability distribution $\rho$ generating $\vec{b}$ must be concentrated on some subset of $X$ with small measure. Firstly, it will be shown that if $\vec{b} \in \mathcal{Q}$ is close to the boundary of $\mathcal{Q}$, then there is a certain set, depending on $\vec{b}$, on which any $\rho\in\mathcal{P}(X)$ with $\int_X\vec{a}\rho \,d\mu=\vec{b}$ must have most of its mass. Then secondly, it will be shown that as $\vec{b}$ approaches the boundary of $\mathcal{Q}$, this set becomes arbitrarily small in measure. These results will also be important in providing a growth bound for the singular potential in the next section.

\begin{lemma}
Let $\vec{b} \in \mathcal{Q}$, $\vec{b}_0\in\partial\mathcal{Q}$ with $|\vec{b}_0-\vec{b}|=\min\limits_{\tilde{\vec{b}}\in\partial\mathcal{Q}}|\tilde{\vec{b}}-\vec{b}|$. Define $\vec{u}\in\mathbb{S}^{k-1}$ to be the unit vector in the direction of $\vec{b}_0-\vec{b}$. Then $\vec{u}$ is normal to a supporting hyperplane of $\mathcal{Q}$ at $\vec{b}_0$, so that $\vec{b}_0\cdot \vec{u}=S_{\vec{u}}$, and $d(\vec{b},\partial\mathcal{Q})=S_{\vec{u}}-\vec{b}\cdot \vec{u}$. Furthermore, there is only one supporting hyperplane of $\mathcal{Q}$ at $\vec{b}_0$, with outward-pointing unit normal vector given by $\vec{u}$.
\end{lemma}
\begin{proof}
First it will be shown that $\vec{u}$ is normal to a supporting hyperplane at $\vec{b}_0$. Assume otherwise for the sake of contradiction. Then there exists some $\vec{u}'\in\mathbb{S}^2$, normal to a supporting hyperplane at $\vec{b}_0$, which must satisfy $\vec{u}\neq \vec{u}'$. Let $\vec{x} \in \mathbb{R}^k$ be the projection of $\vec{b}$ onto the hyperplane given by the set of all $\tilde{\vec{b}} \in \mathbb{R}^k$ with $\tilde{\vec{b}}\cdot \vec{u}'=S_{\vec{u}'}$, In particular, $\vec{x} \not\in\mathcal{Q}$. Furthermore, $|\vec{x}-\vec{b}|<|\vec{b}_0-\vec{b}|$. By considering the straight line from $\vec{b}$ to $\vec{x}$, there must be some $\vec{y} \in \partial\mathcal{Q}$ such that $|\vec{y}-\vec{b}|\leq |\vec{x}-\vec{b}| < |\vec{b}_0-\vec{b}|$, contradicting the assumption that $|\vec{b}_0-\vec{b}|=\min\limits_{\tilde{\vec{b}}\in\partial\mathcal{Q}}|\tilde{\vec{b}}-\vec{b}|$. Then it holds that 
\begin{equation}
\begin{split}
S_{\vec{u}}-\vec{b}\cdot \vec{u} & = (\vec{b}_0-\vec{b})\cdot \vec{u}\\
&=\frac{1}{|\vec{b}_0-\vec{b}|}(\vec{b}_0-\vec{b})\cdot (\vec{b}_0-\vec{b})\\
&=|\vec{b}_0-\vec{b}|=d(\vec{b},\partial\mathcal{Q}).
\end{split}
\end{equation}

To show that this is the only such hyperplane, assume that $\vec{u}_1\in\mathbb{S}^{k-1}$, $\vec{u}_1\neq \vec{u}$ and $\vec{u}_1 \cdot \vec{b}_0 =S_{\vec{u}_1}$. By the same argument as before, by projecting $\vec{b}$ onto the supporting hyperplane with normal $\vec{u}_1$, it contradicts that $\vec{b}_0$ was a closest point on the boundary.\qed
\end{proof}

\begin{proposition}
\label{propConcentrateRho}
 Let $\vec{b} \in \mathcal{Q}$, $\vec{b}_0 \in \partial\mathcal{Q}$ with 
\begin{equation}|\vec{b}-\vec{b}_0|=\min\limits_{\tilde{\vec{b}}\in\partial\mathcal{Q}}|\vec{b}-\tilde{\vec{b}}|,
\end{equation} define $\vec{u}=\frac{1}{|\vec{b}_0-\vec{b}|}(\vec{b}_0-\vec{b})$ and $\epsilon^2=|\vec{b}_0-\vec{b}|$. Let $S_{\vec{u}}$, $E_\epsilon^{\vec{u}}$ be as in Def. \ref{defES}. Then $S_{\vec{u}}-\vec{b}\cdot \vec{u} =\epsilon^2$ and
\begin{equation}
\int_{E_\epsilon^{\vec{u}}}\rho(t)\,d\mu(t) \geq 1-\epsilon
\end{equation}
for all $\rho \in \mathcal{P}(X)$ generating $b$. 
\end{proposition}
\begin{proof}
To see that $S_{\vec{u}}-\vec{b}\cdot\vec{u}=\epsilon^2$ is a result of the previous lemma, and observing that $S_{\vec{u}}-\vec{b}\cdot \vec{u}=(\vec{b}_0-\vec{b})\cdot \vec{u}$. Letting $E=E_\epsilon^{\vec{u}}$ for brevity, from the assumptions it can be seen that for any $\rho\in\mathcal{P}(\mathbb{S}^2)$ that generates $\vec{b}$,
\begin{equation}
\begin{split}
0&=\vec{u}\cdot \vec{b} - S_{\vec{u}}+\epsilon^2\\
&=\int_X \left(\vec{u}\cdot \vec{a}(t)-S_{\vec{u}}+ \epsilon^2\right)\rho(t)\,d\mu(t)\\
&=\int_E \left(\vec{u}\cdot \vec{a}(t)-S_{\vec{u}}+\epsilon^2\right)\rho(t)\,d\mu(t)+\int_{X\setminus E} \left(\vec{u}\cdot \vec{a}(t)-S_{\vec{u}}+\epsilon^2\right)\rho(t)\,d\mu(t)\\
&\leq \epsilon^2 \int_E \rho \,d\mu +(\epsilon^2-\epsilon)\int_{X\setminus E}\rho\,d\mu,
\end{split}
\end{equation}
using that for $t \in X$, $\vec{u}\cdot \vec{a}(t)-S_{\vec{u}}\leq 0$, and for $t \in X\setminus E$ that $ S_{\vec{u}}-\vec{u}\cdot \vec{a}(t)\geq \epsilon$. If $I=\int_E \rho\,d\mu$ for brevity, then this chain of inequalities gives that 
\begin{equation}
\begin{split}
0&\leq\epsilon^2 I +(\epsilon^2-\epsilon)(1-I)
\end{split}
\end{equation}
so that dividing through by $\epsilon$ and rearranging gives the desired result. \qed
\end{proof}

\begin{proposition}
\label{propUniformConvergenceSets}
For $\vec{u} \in\mathbb{S}^{k-1}$ and $\epsilon>0$, let $E^{\vec{u}}_\epsilon\subset X$ be as in Def. \ref{defES}. Then $\lim\limits_{\epsilon \to 0}\mu(E_{\epsilon}^{\vec{u}})=0$, where the convergence is uniform in $\vec{u}$. 
\end{proposition}
\begin{proof}
That $\lim\limits_{\epsilon \to 0} \mu(E_\epsilon^{\vec{u}})=0$ is a consequence of the pseudo-Haar property of the constraint functions. Since $E_\epsilon^{\vec{u}} = (S_{\vec{u}} -\vec{a}\cdot \vec{u})^{-1}\left((-\infty,\epsilon)\right)$, it holds that 
\begin{equation}
\begin{split}
\lim\limits_{\epsilon \to 0}\mu(E_\epsilon^{\vec{u}})&=\mu\left(\bigcap\limits_{\epsilon>0} E_\epsilon^{\vec{u}}\right)\\
&=\mu\left(\left\{t\in X:S_{\vec{u}}-\vec{u}\cdot \vec{a}(t)=0\right\}\right)=0
\end{split}
\end{equation} by the pseudo-Haar property.  

To see that the convergence is uniform, assume otherwise so that there exists some $\gamma > 0$, and $(\vec{u}_j)_{j\in \mathbb{N}}$, $(\alpha_j)_{j \in \mathbb{N}}$ so that the sets $E_j = E_{\alpha_j}^{\vec{u}_j}$ satisfy $\mu(E_j) \geq \gamma$ for all $j$, and $\lim\limits_{j \to \infty}\alpha_j =0$. For brevity, denote the functions $f_j(t) = S_{\vec{u}_j} - \vec{u}_j \cdot \vec{a}(t) - \alpha_j$. Up to a subsequence (not relabelled) it can be assumed that $\vec{u}_j \to \vec{u}^*$, and since the convergence of $\vec{u}_j\cdot \vec{a}(t) \to \vec{u}^*\cdot \vec{a}(t)$ is uniform, there is a corresponding $S^*=S_{\vec{u}^*}=\lim\limits_{j\to\infty}S_{\vec{u}_j}$. Let $\epsilon > 0$. Let $f=\lim\limits_{j\to\infty}f_j$. Then for sufficiently large $j$ since $f_j \to f$ uniformly, it must hold that if $t \in E_j$ then $f(t) >-\epsilon$. This means that $E_j \subset E_{\epsilon}^{\vec{u}^*}$ for sufficiently large $j$. Therefore for large $j$,
\begin{equation}
\begin{split}
\mu(E_j)\leq & \mu(E_\epsilon^{\vec{u}^*})\\
\Rightarrow \limsup\limits_{j \to \infty} \mu(E_j) \leq & \mu(E_\epsilon^{\vec{u}^*}).
\end{split}
\end{equation}
Since $\epsilon$ was arbitrary and $\lim\limits_{\epsilon \to 0}(E_\epsilon^{\vec{u}^*})=0$, this leads to the conclusion that $0 < \gamma \leq \limsup\limits_{j \to\ \infty} \mu(E_j) \leq 0$, a contradiction. \qed
\end{proof}

\section{The Singular Potential}
\label{sectionSingularPotential}
In order to define the singular potential, there are certain constraints on the objective function $\phi$ that are necessary to ensure that the minimisation problem is well posed and the resulting singular potential has desirable properties. Motivated by the Shannon entropy, $\phi(x)=x\ln x$ as the prototypical example, any function satisfying these constraints will be called {\it entropy-like}. 

\begin{definition}
\label{defEntropyLike}
Let $\phi:\mathbb{R}\to(-\infty,\infty]$. Define $\text{dom}(\phi)=\{x \in \mathbb{R}:\phi(x)<+\infty\}$. Then $\phi$ is entropy-like if and only if $\textup{int dom}(\phi)=(0,\infty)$, $\phi|_{(0,\infty)}$ is strictly convex, continuously differentiable, $\lim\limits_{x \to \infty}\frac{\phi(x)}{x}=+\infty$ and $\lim\limits_{x \to 0^+}\phi'(x)=-\infty$.
\end{definition}

In particular any entropy-like function is of Legendre type, so that $\phi'$ is a continuous bijection between $(0,\infty)$ and $\mathbb{R}$ \cite{rockafellar1997convex}. This definition does not place any restriction on the limiting value of $\phi$ at zero. It may hold that $\lim\limits_{x \to 0^+}\phi(x)$ is bounded, such as in the case of the Shannon entropy, or the limit may be infinite, such as for the example $\phi(x)=\frac{1}{x}+x^2$.

\begin{definition}
\label{defSingularPotential}
Let $\phi$ be an entropy-like function. Define the singular potential $\psi_s:\mathcal{Q}\to\mathbb{R}$ by 
\begin{equation}
\psi_s(\vec{b})=\min\limits_{\rho \in \mathcal{A}_{\vec{b}}}\int_X\phi(\rho)\,d\mu
\end{equation}
where the set $\mathcal{A}_{\vec{b}}$ is defined as
\begin{equation}
\mathcal{A}_{\vec{b}}=\left\{\rho \in\mathcal{P}(X):\int_X \rho(t)\vec{a}(t)\,d\mu(t)=\vec{b}\right\}.
\end{equation}
\end{definition}

Using the definition of the singular potential, the next results will be concerned with establishing some key properties of the function. 

\begin{proposition}
\label{propSingularPotentialWellDefined}
Let $\vec{b} \in \mathcal{Q}$. Then the singular potential, $\psi_s(\vec{b})$ is well defined in the sense that the minimisation problem admits solutions. Furthermore, the minimiser is unique, and $\psi_s$ is a strictly convex function. 
\end{proposition}

\begin{proof} The minimisation problem is a convex minimisation problem subject to linear constraints, with constraint functions that are pseudo-Haar. Together with the growth conditions imposed on $\phi$, this ensures the existence of a unique minimising $\rho$, with $\rho$ bounded away from $0$ and $+\infty$, so that $\psi_s(b)<+\infty$ \refcite[Theorem 4.8]{borwein1991duality}. To see convexity, first since $\mathcal{Q}$ is convex there is no problem in taking convex combinations of elements of $\mathcal{Q}$. Let $\vec{b}_1,\vec{b}_2 \in \mathcal{Q}$, $\epsilon \in [0,1]$. If $\rho_1,\rho_2$ solve the minimisation problems corresponding to $\vec{b}_1,\vec{b}_2$ respectively, then $\epsilon \rho_1 + (1-\epsilon)\rho_2 \in \mathcal{A}_{\epsilon \vec{b}_1 +(1-\epsilon)\vec{b}_2}$. Hence 
\begin{equation}
\begin{split}
\psi_s\left(\epsilon \vec{b}_1 + (1-\epsilon)\vec{b}_2\right) &= \min\limits_{\rho \in \mathcal{A}_{\epsilon \vec{b}_1 + (1-\epsilon)\vec{b}_2}} \int_X \phi(\rho(t)) \, d\mu(t)\\
&\leq \int_X \phi\left(\epsilon \rho_1(t) + (1-\epsilon)\rho_2(t)\right)\, d\mu(t)\\
&\leq \int_X \epsilon \phi(\rho_1(t)) + (1-\epsilon)\phi(\rho_2(t)) \, d\mu(t)\\
&= \epsilon\psi_s(\vec{b}_1)+(1-\epsilon)\psi_s(\vec{b}_2).
\end{split}
\end{equation}
Furthermore, note that the strict convexity of $\phi$ gives that for $\vec{b}_1 \neq \vec{b}_2$ and $\epsilon \in (0,1)$, $\rho_1\neq \rho_2$ on a set of positive measure, giving that this inequality is strict. \qed
\end{proof}

\begin{theorem}
There exists some $c>0$ such that for all $\vec{b} \in \mathcal{Q}$, $\vec{b}_0 \in \partial\mathcal{Q}$ with $\min\limits_{\tilde{\vec{b}}\in\partial\mathcal{Q}}|\tilde{\vec{b}}-\vec{b}|=|\vec{b}_0-\vec{b}|=\epsilon^2$ and $\epsilon<c$, the singular potential satisfies the inequality
\begin{equation}
\psi_s(\vec{b}) \geq \mu(E_\epsilon^{\vec{u}})\phi\left(\frac{1-\epsilon}{\mu(E_\epsilon^{\vec{u}})}\right) + \mu(X\setminus E_\epsilon^{\vec{u}})\phi\left(\frac{\epsilon}{\mu(X\setminus E_\epsilon^{\vec{u}})}\right),
\end{equation}
where $E_\epsilon^{\vec{u}}$ is as defined in Def. \ref{defES} and $\vec{u}=\frac{1}{|\vec{b}_0-\vec{b}|}(\vec{b}_0-\vec{b})$.
\end{theorem}
\begin{proof}
Let $\vec{b} \in \mathcal{Q}$, with corresponding $\vec{b}_0,\epsilon,\vec{u}$. Let $E=E_\epsilon^{\vec{u}}$ for brevity. From Prop. \ref{propConcentrateRho} it is known that $\int_E \rho\,d\mu>1-\epsilon$ for all $\rho$ with $\int_X\rho \vec{a}\,d\mu=\vec{b}$. In particular by taking $\rho\in\mathcal{A}_{\vec{b}}$ to be the unique minimiser given in Prop. \ref{propSingularPotentialWellDefined} so that $\psi_s(b)=\int_X\phi(\rho)\,d\mu$, then 
\begin{equation}
\begin{split}
\psi_s(b) &= \int_E \phi(\rho)\,d\mu + \int_{X\setminus E}\phi(\rho)\,d\mu\\
&\geq \mu(E)\phi\left(\frac{1}{\mu(E)}\int_E \rho\,d\mu\right)+\mu(X\setminus E)\phi\left(\frac{1}{\mu(X\setminus E)}\int_{X\setminus E}\rho\,d\mu\right)
\end{split}
\end{equation}
by applying Jensen's inequality. Since $\frac{1}{\mu(E)}\int_E \rho\,d\mu\geq\frac{1-\epsilon}{\mu(E)}$, using the uniform convergence of $\mu(E)\to 0$ in $u$ from Prop. \ref{propUniformConvergenceSets} and that $\phi$ is an increasing function for sufficiently large argument, it must hold that there exists some $c_1>0$ so that if $\epsilon < c_1$ then 
\begin{equation}
\phi\left(\frac{1}{\mu(E)}\int_E\rho\,d\mu\right)\geq \phi\left(\frac{1-\epsilon}{\mu(E)}\right).
\end{equation}
Similarly for $\epsilon < c_2$ since $\phi$ is decreasing for sufficiently small argument, it must hold that 
\begin{equation}\phi\left(\frac{1}{\mu(X\setminus E)}\int_{X\setminus E}\rho\,d\mu\right) \geq \phi\left(\frac{\epsilon}{\mu(X\setminus E)}\right).
\end{equation}
Combining these inequalities gives the lower bound.\qed
\end{proof}

\begin{corollary}
For $j \in \mathbb{N}$ let $\vec{b}_j \in \mathcal{Q}$ with $\lim\limits_{j \to \infty}\inf\limits_{\tilde{\vec{b}} \in \partial \mathcal{Q}}|\tilde{\vec{b}}-\vec{b}_j|=0$. Then $\lim\limits_{j \to \infty} \psi_s(b_j)=+\infty$. Furthermore the convergence is uniform  in $\inf\limits_{\tilde{\vec{b}} \in \partial Q}|\vec{b}-\vec{b}_j|$, so that 
\begin{equation}
\lim\limits_{\epsilon \to 0} \left(\inf\left\{\psi_s(\vec{b}):d(\vec{b},\partial\mathcal{Q})\leq \epsilon\right\}\right)=+\infty.
\end{equation}
\end{corollary}
\begin{proof}
Let $\vec{b} \in \mathcal{Q}$, and $E_\epsilon^{\vec{u}}=E$ for brevity. Using the lower bound for $\psi_s$, and that $\phi$ is bounded below, for some real constant $C$, independent of $\vec{b}$,
\begin{equation}
\psi_s(\vec{b})\geq \mu(E)\phi\left(\frac{1-\epsilon}{\mu(E)}\right)+C=(1-\epsilon)\frac{\phi(\xi)}{\xi}+C,
\end{equation}
where $\xi=\frac{1-\epsilon}{\mu(E)}$. Using that $\mu(E)\to 0$ uniformly as $\vec{b} \to \partial \mathcal{Q}$ and the superlinear growth of $\phi$ gives the required result. \qed
\end{proof}

The precise form of the minimiser can be obtained by considering the dual optimisation problem. For a more thorough account of duality the reader is referred to Ref. \refcite{rockafellar1997convex}, although for completeness a brief heuristic argument will be provided here. 

Let $X$ be some Banach space, $F:X\to\mathbb{R}$ be convex, and $T:X\to \mathbb{R}^k$ be continuous and linear. Let $\vec{b} \in \mathbb{R}^k$ and consider the minimisation problem 
\begin{equation}
\begin{split}
&\min\limits_{x \in X} F(x)\\
\text{subject to}\, \, \, &Tx=\vec{b}.
\end{split}
\end{equation}
Trivially, if $x$ does not satisfy $Tx=\vec{b}$, then $\max\limits_{\lambda \in \mathbb{R}^k} \left(-\lambda \cdot (Tx-\vec{b})\right)=+\infty$, and if $Tx=\vec{b}$, then $(-\lambda\cdot(Tx-\vec{b}))=0$ for all $\lambda \in \mathbb{R}^k$. Therefore the minimisation problem can be written as 
\begin{equation}\label{equationDuality}
\min\limits_{x \in X}\max\limits_{{\bf\lambda} \in \mathbb{R}^k} F(x)-\lambda\cdot(Tx-\vec{b}).
\end{equation}
Let $X^*$ be the dual space to $X$ with duality pairing $\langle \cdot,\cdot\rangle$, $T^*:\mathbb{R}^k \to X^*$ be the adjoint of $T$, and $F^*:X^*\to\mathbb{R}$ be the convex conjugate of $F$ defined by $F^*(x^*)=\sup\limits_{x \in X} \langle x^*,x\rangle-F(x)$. Assuming that the minimum and maximum in \eqref{equationDuality} can be exchanged, this minimisation problem can be written as 
\begin{equation}
\begin{split}
\max\limits_{\lambda \in \mathbb{R}^k}\min\limits_{x \in X} F(x)-\lambda\cdot(Tx-\vec{b}) &= \max\limits_{\lambda \in \mathbb{R}^k}\lambda \cdot \vec{b} -\left(\max\limits_{x \in X} \langle T^*\lambda, x\rangle - F(x)\right)\\
&=\max\limits_{\lambda \in \mathbb{R}^k} \lambda \cdot \vec{b} - F^*(T^*\lambda).
\end{split}
\end{equation}
If the interchange of minimum and maximum is permitted, then the dual problem 
\begin{equation}
\max\limits_{\lambda \in \mathbb{R}^k}\lambda \cdot \vec{b} - F^*(T^*\lambda)
\end{equation}
can provide information about the original so-called primal problem. Furthermore, since the minimisation is over a finite dimensional set without constraints, it may be easier to approach. This is what is done in the work of Borwein and Lewis \cite{borwein1991duality}, and their results will be exploited for the sake of this work. In particular the minimiser of the primal problem can be described in terms of maximiser of the dual problem. 

\begin{proposition}
\label{propFormRho}
Let $\phi: \mathbb{R} \to (-\infty,\infty]$ be an entropy-like function. Then for $\vec{b} \in \mathcal{Q}$, the unique minimiser of 
\begin{equation}
\min\limits_{\rho \in \mathcal{A}_{\vec{b}}} \int_X \phi(\rho) \, d\mu
\end{equation}
is achieved by 
\begin{equation}
\rho_{\vec{b}}(t) = (\phi')^{-1}\left(\alpha + \lambda \cdot \vec{a}(t)\right),
\end{equation}
where $\alpha \in \mathbb{R}$, $\lambda \in \mathbb{R}^k$ are the unique maximisers of the dual optimisation problem,
\begin{equation}
\max\limits_{\tilde{\alpha},\tilde{\lambda}}\, \tilde{\alpha} + \tilde{\lambda} \cdot \vec{b} - \int_X \phi^*\left(\tilde{\alpha} + \tilde{\lambda} \cdot \vec{a}(t)\right)\,d\mu.
\end{equation}
\end{proposition}
\begin{proof}
This is a straightforward application of Ref. \refcite[Theorem 4.8]{borwein1991duality}, .\qed
\end{proof}

The next stage is to establish smoothness properties of the singular potential. In order to show differentiability of $\psi_s$, the key point is to observe that the map from the Lagrange multipliers to the moments $\vec{b}$ is explicitly given and invertible. As such the differentiability of $\psi_s$ is inherited from the explicit map from the Lagrange multipliers to $\mathcal{Q}$.

\begin{proposition}
\label{propSingularPotSmooth} Let $\phi: \mathbb{R} \to (-\infty,\infty]$ be entropy-like, $C^m$ on $(0,\infty)$ for $m \geq 2$, with $\phi''$ positive everywhere. Then the map $\vec{b} \mapsto(\alpha,\lambda)$ is $C^{m-1}$, where $\alpha,\lambda$ solve the dual optimisation problem.
\end{proposition}
\begin{proof}
For $(\alpha,\lambda) \in \mathbb{R} \times \mathbb{R}^k$, consider the map
\begin{equation}
h:(\alpha,\lambda) \mapsto \int_X \left(1,\vec{a}(t)\right)\left(\phi'\right)^{-1}\left(\alpha + \lambda \cdot \vec{a}(t) \right) \, d\mu(t) = (\beta,\vec{b}).
\end{equation}
Since $\phi''$ is positive, the inverse function theorem can be applied to $\left(\phi'\right)^{-1}$ , and the differentiability of $\phi$ gives that $\left(\phi'\right)^{-1}$ is $C^{m-1}$ on $\mathbb{R}$. In particular, this implies that $h \in C^{m-1}\left(\mathbb{R}^{k+1},\mathbb{R}^{k+1}\right)$. The Jacobian matrix of $h$ is given by 
\begin{equation}
\begin{split}
\label{eqdLamdB}
\frac{\partial h(\alpha,\lambda)}{\partial (\alpha,\lambda)} &= \int_X (1,\vec{a}(t))\otimes (1,\vec{a}(t)) \left(\phi''\circ \left(\phi'\right)^{-1}\left(\alpha+\lambda \cdot \vec{a}(t)\right) \right)^{-1} \, d\mu(t).
\end{split}
\end{equation}
Here $\circ$ denotes function composition. If this Jacobian matrix failed to be invertible, there would exist some $\xi \in \mathbb{R}^{k+1}$ with $\xi\neq 0$ so that 
\begin{equation}
0 = \int_X \left((1,\vec{a}(t))\cdot \xi\right)^2 \left(\phi''\circ \left(\phi'\right)^{-1}\left(\alpha+\lambda \cdot \vec{a}(t)\right) \right)^{-1} \, d\mu(t).
\end{equation}
Since $\phi''$ is positive everywhere, this implies that $\left(\xi \cdot (1,\vec{a}(t))\right)^2=0$ almost everywhere in $X$; however by the pseudo-Haar property this cannot hold, a contradiction. By applying the inverse function theorem to $h$, this gives that the map 
\begin{equation}
\left(\beta,\vec{b}\right) \mapsto (\alpha,\lambda)
\end{equation}
is $C^{m-1}$ for $(\beta,\vec{b})\in h(\mathbb{R}^{k+1})$. In particular, $h^{-1}(1,\vec{b})$, which maps $\vec{b}\in\mathcal{Q}$ to $(\alpha,\lambda)$ is also a $C^{m-1}$ function. Global invertibility is not problematic, since the existence and uniqueness of solutions for dual problem ensures the existence and unique pair $(\alpha,\lambda)$ for each $\vec{b} \in \mathcal{Q}$, and non-existence otherwise.\qed
\end{proof}

\begin{proposition}
\label{propSingularFirstDerivative}
Let $\phi$ be a $C^2$ entropy-like function with $\phi''$ positive. Then for $\vec{b} \in \mathcal{Q}$, $\frac{\partial \psi_s}{\partial \vec{b}}(\vec{b})=\lambda$, where $\lambda$ solves the dual optimisation problem corresponding to $\vec{b}$. 
\end{proposition}

\begin{proof}
From Ref. \refcite{borwein1991duality} it is possible to write the singular potential in terms of the Lagrangian dual optimisation problem, so that for dual optimal $\alpha,\lambda$, the singular potential can be written as 
\begin{equation}
\psi_s(\vec{b})= \alpha + \lambda\cdot \vec{b} - \int_X \phi^*(\alpha+\lambda\cdot \vec{a}(t))\,d\mu(t)
\end{equation}
Using from Prop. \ref{propSingularPotSmooth} that the map $\vec{b}\mapsto (\alpha,\lambda)$ is at least $C^1$, differentiating through gives that 
\begin{equation}
\begin{split}
\frac{\partial \psi_s}{\partial \vec{b}} &=\frac{\partial \alpha}{\partial \vec{b}} + \frac{\partial \lambda}{\partial \vec{b}}\cdot \vec{b} +\lambda- \int_X (\phi')^{-1}\left(\alpha + \lambda \cdot \vec{a}(t)\right) \left( \frac{\partial \alpha}{\partial \vec{b}}+\frac{\partial \lambda}{\partial \vec{b}}\cdot \vec{a}(t) \right) \,d\mu(t)\\
&=\lambda + \frac{\partial \alpha}{\partial \vec{b}}\left(1-\int_X\rho_{\vec{b}}(t)\,d\mu(t)\right)+\frac{\partial \lambda}{\partial \vec{b}}\cdot\left(\vec{b}-\int_X \rho_{\vec{b}}(t) \vec{a}(t)\,d\mu(t)\right)\\
&=\lambda .
\end{split}
\end{equation}\qed
\end{proof}

\begin{corollary}
\label{corollarySingularCM}
If $\phi$ is entropy-like, $C^m$ for $m\geq 2$ on $(0,\infty)$ and $\phi''$ is positive, then $\psi_s\in C^m(\mathcal{Q},\mathbb{R})$. 
\end{corollary}
\begin{proof}
Combining the results of Prop. \ref{propSingularPotSmooth} and \ref{propSingularFirstDerivative}, it holds that 
\begin{equation}\frac{\partial \psi_s}{\partial \vec{b}}=\lambda \in C^{m-1}(\mathcal{Q};\mathbb{R}^k),
\end{equation}
so that $\psi_s \in C^m(\mathcal{Q},\mathbb{R})$.\qed
\end{proof}

\begin{proposition}
\label{propSecondDerivative}
Assume $\phi$ is entropy-like, $C^m$ for $m\geq 2$ on the interior of its domain and that $\phi''$ is positive. Let $\vec{b}_0 \in \mathcal{Q}$, and denote $\rho=\rho_{\vec{b}_0}$, $\sigma = \frac{1}{\phi''(\rho)}$, $Z=\int_X \sigma \,d\mu$. and $\tilde{\sigma}=\frac{1}{Z}\sigma$. Then 
\begin{equation}
\frac{\partial\lambda}{\partial \vec{b}}(\vec{b}_0) = \frac{1}{Z}\left(\int_X \vec{a} \otimes \vec{a} \tilde{\sigma}\,d\mu-\int_X \vec{a} \tilde{\sigma}\,d\mu\otimes \int_X \vec{a}\tilde{\sigma}\,d\mu\right)^{-1}.
\end{equation}
In particular, if $\phi$ is the Shannon entropy $\phi(x)=x\ln x$, then 
\begin{equation}
\frac{\partial\lambda}{\partial \vec{b}}(\vec{b}_0) =\left(\int_X \vec{a} \otimes \vec{a} \rho\,d\mu-\vec{b}_0\otimes \vec{b}_0\right)^{-1}.
\end{equation}
\end{proposition}
\begin{proof}
Under the standing assumptions $\alpha,\lambda$ as functions of $\vec{b}$ are at least $C^1$. Following a similar argument to Prop. \ref{propSingularPotSmooth}, 
\begin{equation}
\begin{split}
1 &= \int_X \left(\phi'\right)^{-1}(\alpha+\lambda\cdot\vec{a})\,d\mu\\
\Rightarrow 0 &= \int_X \left(\frac{\partial \alpha}{\partial b_i} + \frac{\partial \lambda}{\partial b_i}\cdot \vec{a}\right)\sigma\,d\mu\\
&=Z\frac{\partial \alpha}{\partial b_i} + \int_X \vec{a}\sigma \,d\mu \cdot \frac{\partial \lambda}{\partial b_i}.
\end{split}
\end{equation}
Similarly, considering the relation between $\alpha,\lambda$ and $\vec{b}$ gives 
\begin{equation}
\begin{split}
b_i &= \int_X a_i \left(\phi'\right)^{-1}(\alpha+\lambda\cdot \vec{a})\,d\mu\\
\Rightarrow \delta_{ij} &= \int_X \left(a_i \frac{\partial \alpha}{\partial b_j} + a_i \frac{\partial \lambda}{\partial b_j}\cdot \vec{a}\right) \sigma\,d\mu\\
&=-\frac{1}{Z}\int_X a_i\sigma\,d\mu\int_X\vec{a}\sigma\,d\mu\cdot \frac{\partial \lambda}{\partial b_j}+\int_X a_i a\sigma\,d\mu\cdot \frac{\partial \lambda}{\partial b_j}\\
\Rightarrow I &=Z\left(\int_X \vec{a} \otimes \vec{a} \tilde{\sigma}\,d\mu - \int_X \vec{a} \tilde{\sigma}\,d\mu \otimes \int_X \vec{a}\tilde{\sigma}\,d\mu\right)\frac{\partial\lambda}{\partial \vec{b}}.
\end{split}
\end{equation}
Since $\phi''$, and therefore $\sigma$ is positive, and $\tilde{\sigma}=\frac{1}{\int_X \sigma \,d\mu}\sigma$, this gives that $\tilde{\sigma} \in \mathcal{P}(X)$, and furthermore due to the bounds on $\rho$, $\tilde{\sigma}$ is bounded away from zero. Therefore by use of the pseudo-Haar condition analogously to in Prop. \ref{propSingularPotSmooth}, the matrix 
\begin{equation}
Z\left(\int_X \vec{a} \otimes \vec{a} \tilde{\sigma}\,d\mu - \int_X \vec{a} \tilde{\sigma}\,d\mu \otimes \int_X \vec{a}\tilde{\sigma}\,d\mu\right)
\end{equation}
is invertible, so that taking its inverse gives the result. To see the case when $\phi$ is the Shannon entropy follows immediately since $\phi''(x)=\frac{1}{x}$, which gives $\sigma=\rho$, so $Z=1$ and $\tilde{\sigma}=\rho$.\qed
\end{proof}

\begin{proposition}
\label{propBToRhoContinuous}
Let $\phi: \mathbb{R} \to (-\infty,\infty]$ be an entropy-like function, $C^2$ with $\phi''$ positive on $(0,\infty)$. Then the map $F:\mathcal{Q} \to \mathcal{P}(X)$ given by $F(\vec{b})= \rho_{\vec{b}}$ is continuous with respect to the $L^\infty$ topology on $\mathcal{P}(X)$. 
\end{proposition} 

\begin{proof}
From the argument in Prop. \ref{propSingularPotSmooth} it is sufficient to show that the map $G:(\alpha,\lambda) \mapsto \left(\phi'\right)^{-1}(\alpha + \lambda \cdot \vec{a})$ is continuous as a function from $\mathbb{R}^{k+1}$ to $L^\infty$. That $G(\alpha,\lambda) \in L^\infty$ is a consequence of $\left(\phi'\right)^{-1}$ being continuous and the inequality $|\alpha+\lambda\cdot \vec{a}(t) |<|\alpha| + |\lambda| \, ||a||_\infty$. To show continuity, fix $\alpha_0,\lambda_0$. Define $K = |\alpha_0| + |\lambda_0|\,||\vec{a}||_\infty$. Consider $\alpha,\lambda$ with $|\alpha-\alpha_0| + |\lambda - \lambda_0| \, ||\vec{a}||_\infty < \epsilon$ for $\epsilon >0$. Then 
\begin{equation}
\begin{split}
|\alpha + \lambda \cdot \vec{a}(t) | &= |\alpha - \alpha_0 + (\lambda - \lambda_0)\cdot \vec{a}(t) + \alpha_0 + \lambda_0\cdot \vec{a}(t) |\\
&\leq \epsilon + K.
\end{split}
\end{equation}
Since $\left(\phi'\right)^{-1}$ is $C^1$ on $\mathbb{R}$, it is Lipschitz on $[-(K+\epsilon),K+\epsilon]$ with Lipschitz constant $L$, say. Hence 
\begin{equation}
\begin{split}
\left|\left(\phi'\right)^{-1}(\alpha+\lambda \cdot \vec{a}(t)) - \left(\phi'\right)^{-1}(\alpha_0 + \lambda_0\cdot(\vec{a}(t))\right| &\leq L |\alpha-\alpha_0 + (\lambda - \lambda_0) \cdot \vec{a}(t)|\\
&\leq L\left(|\alpha-\alpha_0| + |\lambda-\lambda_0 |\, ||\vec{a}||_\infty\right)\\
&\leq L\epsilon.
\end{split}
\end{equation}\qed
\end{proof}

\section{Applications}
\label{sectionApplications}
For the remainder of this section, $\phi$ will be assumed to be entropy-like and $C^2$ with positive second derivative on $(0,\infty)$. In particular from Corollary \ref{corollarySingularCM} this implies $\psi_s$ is $C^2$. Given $\vec{b} \in \mathcal{Q}$, $\rho_{\vec{b}}$ will be as in Proposition \ref{propFormRho}, and given $\rho \in \mathcal{P}(X)$, define $\vec{b}_\rho=\int_X \vec{a}\rho\,d\mu$. Within this section all the examples of constraint functions can be identified with linearly independent sets of analytic functions on connected subsets of $\mathbb{R}^n$, with corresponding measure absolutely continuous with respect to the Lebesgue measure, so that by \cite{lewis1995consistencyReport}, the pseudo-Haar condition is satisfied.  

\subsection{Mean Field and Nonlinear Constraint Models}
\label{SubsecMeanField}
The aim of this section is to establish some elementary results on minimisation problems for certain functionals on $\mathcal{P}(X)$ by reducing the problem to a finite dimensional problem via the machinery of the singular potential. The key point is to establish equivalence between global minima, local minima and critical points of functionals on $\mathcal{P}(X)$ and global minima, local minima and critical points (respectively) of functions on $\mathcal{Q}$. Once this has been shown it is possible to establish standard results such as existence of minimisers by using simple tools of real analysis rather than having to resort to more complicated arguments based on functional analysis. Critical points of the functionals on $\mathcal{P}(X)$ will be expressed as implicit relations on the dual variables. This loosely provides a framework for dealing with such minimisation problems, whereby the model is derived as a problem in $\mathcal{P}(X)$ by some physical argument, then the singular potential allows analytical questions to be asked in the simpler language of real analysis as a problem in $\mathcal{Q}$, and then finally the dual variables provide a framework that is attractive from the point of view of numerical analysis.

Let $\emptyset \neq A \subset \mathcal{Q}$ be relatively closed in $\mathcal{Q}$. Let $f:\mathcal{Q}\to\mathbb{R}$ be continuous on $A$ and bounded from below, with $f|_{\mathcal{Q}\setminus A}=+\infty$. Define the functionals $\mathcal{I}_{\mathcal{P}}:\mathcal{P}(X)\to\mathbb{R}$ and $\mathcal{I}_{\mathcal{Q}}:\mathcal{Q}\to\mathbb{R}$ by 
\begin{equation}
\begin{split}
\mathcal{I}_{\mathcal{P}}(\rho)&=\int_X\phi(\rho)\,d\mu+f(\vec{b}_\rho),\\
\mathcal{I}_{\mathcal{Q}}(\vec{b})&=\psi_s(\vec{b})+f(\vec{b})=\mathcal{I}_{\mathcal{P}}(\rho_{\vec{b}}).
\end{split}
\end{equation}
Consider the minimisation problems (P1) and (P2) given by
\begin{equation}
\begin{split}
\text{(P1)}\,\,\,\,\, &\min\limits_{\rho \in \mathcal{P}(X)}\mathcal{I}_{\mathcal{P}}(\rho)\\
\text{(P2)}\,\,\,\,\, &\min\limits_{\vec{b} \in \mathcal{Q}}\mathcal{I}_{\mathcal{Q}}(\vec{b}).
\end{split}
\end{equation}
The aim of this section is to show equivalence of global minima, local minima and critical points for (P1) and (P2) and characterise such points. For applications (P1) will be obtained through a physical argument, and (P2) loosely speaking is a simpler, macroscopic, equivalent model. The motivation for the problem (P1)  comes from two different areas. Firstly one can see maximum entropy methods subject to nonlinear constraints as an example. Let $g:\mathcal{Q} \to \mathbb{R}$ be a continuous function. Consider the minimisation problem 
\begin{equation}
\begin{split}
\min\limits_{\rho \in \mathcal{P}(X)} \int_X \phi(\rho)\,d\mu\\
\text{subject to } g(\vec{b}_\rho)=0
\end{split}
\end{equation}

Define $f(\vec{b})=0$ if $g(\vec{b})=0$ and $f(\vec{b})=+\infty$ if $g(\vec{b})\neq 0$ so that $A=\{\vec{b} \in \mathcal{Q}:g(\vec{b})=0\}$, which by the continuity of $g$ is relatively closed. Then the minimisation problem 
$$\min\limits_{\rho \in \mathcal{P}(X)} \int_X \phi(\rho)\,d\mu + f(\vec{b}_\rho)$$
is equivalent to the nonlinear constraint model. Examples from the literature of such models are seen in statistical models of isotropic elasticity \cite{treloar1975physics}, where $X=\mathbb{S}^2$, $a_i(p)=p_i$ (in Cartesian coordinates) for $\vec{p}\in\mathbb{S}^2$ represents the orientation of molecules in a polymer chain and $b$ corresponds to the end to end vector spanned by the entire chain. The constraint in this case is of the form $|\vec{b}| = r$ for $r \in [0,1)$. Maximum entropy methods with to nonlinear constraints were also considered by Decarreau {\it et.al.}\cite{decarreau1992dual} as a method for dealing with the phase problem in crystallography. In diffraction experiments it is much simpler to observe the intensity of a wave than its phase, and mathematically this corresponds to having knowledge of the modulus of a Fourier coefficient representing the wave whilst its argument is unknown. For example, Decarreau {\it et al.} consider minimisation problems such as 
\begin{displaymath}
\left\{\begin{array}{l}
\min\limits_{\rho \in \mathcal{P}(\Omega)}\int_\Omega \phi(\rho(x))\,dx\\
\left(\int_\Omega \cos(n\cdot x)\rho(x)\,dx\right)^2 +\left(\int_\Omega \sin(n \cdot x)\rho(x)\,dx\right)^2 = m_n^2\,\,\,\forall n \in N
\end{array}\right. ,
\end{displaymath}
where $\Omega=[0,2\pi]$, $N$ is some finite subset of $\mathbb{N}$ and $m_n$ are given real numbers.

A second example of problems of the form (P1) are given by the mean field approximation. Given some state space $X$ and constraint functions $(a_i)_{i=1}^k$, the free energy is typically of the form 
\begin{equation}
T\int_X \phi(\rho)\,d\mu -\frac{1}{2}\tens{K}\vec{b}_\rho\cdot \vec{b}_\rho -\vec{H}\cdot \vec{b}_\rho
\end{equation}
where $T>0$ represents temperature, $\tens{K}$ is a positive definite $k\times k$ matrix that represents some kind of interaction potential, and $\vec{H} \in \mathbb{R}^k$ is representative of some kind of external influence such as magnetic/electric fields. In practice the Shannon entropy $\phi(x)=x \ln x$ is almost exclusively used, although for this analysis this is not necessary. 

In the literature these minimisation problems often are not dealt with rigorously, and in particular the issue of non-differentiability of the objective function $\phi$ at $0$ is often neglected, so that the given solution is obtained by considering only points where the first variation of the free energy is zero. The following results will aim to show that under rather non-restrictive assumptions the minimisation problems are well posed and the solutions behave as one would expect with a non-rigorous analysis. Through use of the singular potential the minimisation problem can be reduced to an analytically simpler but equivalent problem (P2) in finite dimensions, for which one can easily obtain the desired results.

\begin{proposition}
\label{propGlobalMinsEquiv}
There exists global minimisers $\vec{b}^*$ and $\rho^*$ of $\mathcal{I}_{\mathcal{Q}}$ and $\mathcal{I}_{\mathcal{P}}$ respectively, and all global minimisers of $\mathcal{I}_{\mathcal{Q}}$ are in one-to-one correspondence with global minimisers of $\mathcal{I}_{\mathcal{P}}$ via the map $\vec{b}\mapsto \rho_{\vec{b}}$. In particular, global minimisers of $\mathcal{I}_{\mathcal{P}}$ are bounded away from zero and infinity.
\end{proposition}

\begin{proof}
From the assumptions on $\phi,f$, it is immediate that these functionals are bounded below. Furthermore, their infima coincide, since 
\begin{equation}
\begin{split}
\inf\limits_{\rho \in \mathcal{P}(X)}\mathcal{I}_{\mathcal{P}}(\rho) &= \inf\limits_{\vec{b} \in \mathcal{Q}}\left(\inf\limits_{\rho \in \mathcal{A}_{\vec{b}}} \mathcal{I}_{\mathcal{P}}(\rho)\right)\\
&=\inf\limits_{\vec{b} \in \mathcal{Q}}\mathcal{I}_{\mathcal{Q}}(\vec{b}).
\end{split}
\end{equation}
Since $\mathcal{Q}$ is a precompact set, $\mathcal{I}_{\mathcal{Q}}$ is continuous on the relatively closed set $A$, and $\lim\limits_{\vec{b} \to \partial\mathcal{Q}}\mathcal{I}_{\mathcal{Q}}(\vec{b})=+\infty$, it must hold that a minimum exists for $\mathcal{I}_{\mathcal{Q}}$, with corresponding minimiser $\vec{b}^*$. Furthermore, since $\mathcal{I}_{\mathcal{P}}(\rho_{\vec{b}^*})=\mathcal{I}_{\mathcal{Q}}(\vec{b}^*)$, the minimum of $\mathcal{I}_{\mathcal{P}}$ must be attained at $\rho_{\vec{b}^*}$. Finally, if $\rho^*$ is a global minimiser for $\mathcal{I}_{\mathcal{P}}$, then it must hold that $\rho^*=\rho_{\vec{b}^*}$ for $\vec{b}^*=\vec{b}_{\rho^*}$, since otherwise $\mathcal{I}_{\mathcal{P}}(\rho^*)>\mathcal{I}_{\mathcal{Q}}(\vec{b}^*)=\mathcal{I}_{\mathcal{P}}(\rho_{\vec{b}^*})$. \qed
\end{proof}

\begin{proposition}
\label{propLocalMins}
There is a one-to-one correspondence between local minimisers of $\mathcal{I}_{\mathcal{Q}}$ and $L^1$-local minimisers of $\mathcal{I}_{\mathcal{P}}$, given by the map $\vec{b}\mapsto \rho_{\vec{b}}$, so that $\rho^* \in \mathcal{P}(X)$ is an $L^1$-local minimiser if and only if $\vec{b}_{\rho^*} \in \mathcal{Q}$ is a local minimiser, and vice versa. In particular, all local minimisers of $\mathcal{I}_{\mathcal{P}}$ are bounded away from zero and infinity. The equivalence also holds for strict local minimisers.
\end{proposition}
\begin{proof}
The proof for strict local minimisers is identical to that for non-strict minimisers, so the proof for strict local minimisers will be omitted. Let $\vec{b}^* \in \mathcal{Q}$ be a local minimiser, so that $\mathcal{I}_{\mathcal{Q}}(\vec{b}^*)\leq \mathcal{I}_{\mathcal{Q}}(\vec{b})$ for all $\vec{b} \in \mathcal{Q}$ with $|\vec{b}-\vec{b}^*|<\epsilon$. Define $\rho^* =\rho_{\vec{b}^*}$. Let $\rho \in \mathcal{P}(X)$ with $||\rho-\rho^*||_1<\delta$. Then if $\vec{b}=\vec{b}_\rho$, $|\vec{b}-\vec{b}^*|<\delta||\vec{a}||_\infty $. Finally, 
\begin{equation}
\begin{split}
\mathcal{I}_{\mathcal{P}}(\rho)-\mathcal{I}_{\mathcal{P}}(\rho^*) & = \int_X \phi(\rho)\,d\mu + f(\vec{b})-\mathcal{I}_{\mathcal{Q}}(\vec{b}^*)\\
&\geq \psi_s(\vec{b}) + f(\vec{b}) -\mathcal{I}_{\mathcal{Q}}(\vec{b}^*)\\
&=\mathcal{I}_{\mathcal{Q}}(\vec{b})-\mathcal{I}_{\mathcal{Q}}(\vec{b}^*).
\end{split}
\end{equation}
Therefore if $\delta < \frac{\epsilon}{||\vec{a}||_\infty}$, it holds that $|\vec{b}-\vec{b}^*|<\epsilon$ so that consequently $\mathcal{I}_{\mathcal{P}}(\rho)\geq \mathcal{I}_{\mathcal{P}(X)}(\rho^*)$.

To show the converse statement, note that if $\rho^*$ is an $L^1$-local minimiser of $\mathcal{I}_{\mathcal{P}}$, then $\rho^*=\rho_{\vec{b}^*}$. This can be seen by considering $\rho=(1-\gamma)\rho^*+\gamma\rho_{\vec{b}^*}$ for $\gamma>0$ small, and using the strict convexity of $\phi$, which gives that $\mathcal{I}_{\mathcal{P}(X)}(\rho)< \mathcal{I}_{\mathcal{P}(X)}(\rho^*)$.

Now let $\rho^*$ be an $L^1$-local minimiser, so that for all $\rho \in \mathcal{P}(X)$ with \linebreak $||\rho-\rho^*||_1<\epsilon$, $\mathcal{I}_{\mathcal{P}}(\rho)\geq\mathcal{I}_{\mathcal{P}}(\rho^*)$. Now consider $\vec{b} \in \mathcal{Q}$. Then
\begin{equation}
\mathcal{I}_{\mathcal{Q}}(\vec{b})-\mathcal{I}_{\mathcal{Q}}(\vec{b}^*) = \mathcal{I}_{\mathcal{P}}(\rho_{\vec{b}})-\mathcal{I}_{\mathcal{P}}(\rho^*).
\end{equation}
All that remains to show is that there exists some $\delta>0$ so that for all $\tilde{\vec{b}} \in \mathcal{Q}$ with $|\tilde{\vec{b}}-\vec{b}^*| < \delta$, then $||\rho_{\vec{b}}-\rho^*||_1<\epsilon$. This holds, since from Prop. \ref{propBToRhoContinuous} it is known that the map $\vec{b}\mapsto \rho_{\vec{b}}$ is continuous in $L^\infty$, and hence is continuous in $L^1$.\qed
\end{proof}

For the following result, a critical point of a function $F:\mathcal{P}(X)\to\mathbb{R}$ is defined to be any $\rho$ in $\mathcal{P}(X)$ that is bounded away from zero, such that for all $\xi \in L^\infty(X)$ with $\int_X \xi\,d\mu=0$, 
\begin{equation}
\left.\frac{d}{d \tau}  F(\rho+\tau\xi)\right|_{\tau=0}=0.
\end{equation}
\begin{proposition}
\label{propCriticalPointsEquivalent}
Let $\textup{int}(A) \neq \emptyset$ and let $f$ be $C^1$ on $\textup{int}(A)$. Then $\vec{b}^* \in \textup{int}(A)$ is a critical point of $\mathcal{I}_{\mathcal{Q}}$ if and only if $\rho^*=\rho_{\vec{b}^*}$ is a critical point of $\mathcal{I}_{\mathcal{P}}$ in $\{\rho \in \mathcal{P}(X): \vec{b}_\rho \in \textup{int}(A)\}$.
\end{proposition}
\begin{proof}
Since $\vec{b}^*$ is a critical point of the $C^1$ function $\mathcal{I}_{\mathcal{Q}}$, it holds that 
\begin{equation}
\begin{split}
0&= \nabla \mathcal{I}_{\mathcal{Q}}(\vec{b}^*)\\
&=\lambda(\vec{b}^*)+\nabla f(\vec{b}^*).
\end{split}
\end{equation}
This implies that $\phi'(\rho^*)=\alpha - \nabla f(\vec{b}^*)\cdot \vec{a}$. Consider any $\xi \in L^\infty(X)$ with $\int_X \xi \,d\mu=0$. This gives
\begin{equation}
\begin{split}
\left.\frac{d}{d\tau}\mathcal{I}_{\mathcal{P}}(\rho^*+\tau\xi)\right|_{\tau=0} & = \int_X \left(\phi'(\rho^*)-\nabla f(\vec{b}^*)\cdot a \right)\xi\,d\mu\\
&= \alpha\int_X \xi \, d\mu\\
&=0.
\end{split}
\end{equation}
The converse follows by the same argument.\qed
\end{proof}

\begin{remark}
Whilst Proposition \ref{propLocalMins} shows that there is a one-to-one correspondence between local minima of $\mathcal{I}_{\mathcal{P}}$ and $\mathcal{I}_{\mathcal{Q}}$, an analogous result does not hold for local maxima. Let $\vec{b}\in\mathcal{Q}$ be a local maximum for $\mathcal{I}_{\mathcal{Q}}$, and let $\xi \in L^\infty(X)\setminus\{0\}$ be such that $\int_X \xi \,d\mu=0$ and $\int_X \vec{a}\xi\,d\mu=\vec{0}$. Then $\rho_{\vec{b}}+\tau\xi \in \mathcal{A}_{\vec{b}}$ for $\tau \in \mathbb{R}$ sufficiently small to ensure the non-negativity constraint is satisfied, and hence 
\begin{equation}
\begin{split}
\mathcal{I}_{\mathcal{P}}(\rho_{\vec{b}}+\tau \xi ) &= \int_{\mathbb{S}^2}\phi(\rho_{\vec{b}}+\tau\xi) +f(\vec{b})\\
&>\psi_s(\vec{b})+f(\vec{b})\\
&=\mathcal{I}_{\mathcal{Q}}(\vec{b})\\
&=\mathcal{I}_{\mathcal{P}}(\rho_{\vec{b}}).
\end{split}
\end{equation}
By taking $\tau$ sufficiently small shows that $\rho_{\vec{b}}$ is not a local maximum, although if $f\in C^1$, $\rho_{\vec{b}}$ will be a critical point by Prop. \ref{propCriticalPointsEquivalent}. 
\end{remark}

\begin{corollary}\label{corollaryCriticalPointCondition1}
Consider the minimisation problem $(\textup{P}1)$. If $A=\mathcal{Q}$ and $f \in C^1(\overline{\mathcal{Q}};\mathbb{R})$ then there exists a global minimum, and all $L^1$-local minima and critical points satisfy
\begin{equation}
\rho^*(t)=\left(\phi'\right)^{-1}\left(\alpha - \nabla f(\vec{b}_\rho)\cdot \vec{a}(t)\right)
\end{equation}
for some constant $\alpha \in \mathbb{R}$. 
\end{corollary}
\begin{proof}
This is a straightforward application of Propositions  \ref{propGlobalMinsEquiv}, \ref{propLocalMins} and \ref{propCriticalPointsEquivalent}.\qed
\end{proof}

\begin{corollary}\label{corollaryCriticalPointCondition2}
Consider the minimisation problem $(\textup{P}1)$. Let $g\in C^1(\mathcal{Q},\mathbb{R})$ and $A=\{\vec{b}\in\mathcal{Q}:g(\vec{b})=0\}\neq \emptyset$. Then there exists a global minimum, and all $L^1$-local minima and critical points satisfy 
\begin{equation}
\rho^*(t)=\left(\phi'\right)^{-1}\left(\alpha+ \eta\nabla g(\vec{b})\cdot \vec{a}(t)\right)
\end{equation}
for all $t \in X$ and some Lagrange multiplier $\eta \in \mathbb{R}$.
\end{corollary}
\begin{proof}
Again this is a straightforward application of Propositions   \ref{propGlobalMinsEquiv}, \ref{propLocalMins} and \ref{propCriticalPointsEquivalent}.\qed
\end{proof}

Although the equivalent problem (P2) is simpler by virtue of being finite-dimensional, the non-explicit representation of the energy limits its usefulness. However, an order parameter that is in some sense dual to $\vec{b}\in\mathcal{Q}$ can be considered, and reduces the dual optimisation scheme needed at each point in the domain to a simple 1-dimensional problem. First, a lemma is required. 

\begin{lemma}
Let $\lambda \in \mathbb{R}^k$. Then there exists a unique $\alpha_\lambda \in\mathbb{R}$ such that $\int_X \left(\phi'\right)^{-1}(\alpha_\lambda +\lambda \cdot a(t))\,d\mu(t)=1$, which can be found by solving 
\begin{equation}
\max\limits_{\alpha \in\mathbb{R}} \alpha -\int_X \phi^*(\alpha+\lambda\cdot a(t))\,d\mu(t).\label{eqMax1}
\end{equation}
Furthermore, if $\phi$ is $C^m$ on its effective domain for $m\geq 2$ then $\alpha$ is a $C^{m-1}$ function of $\lambda$.
\end{lemma}
\begin{proof}
Existence of such an $\alpha_\lambda$ can be seen by recalling that $\left(\phi'\right)^{-1}$ is an increasing bijection between $\mathbb{R}$ and $(0,\infty)$ and the map $\alpha \mapsto \int_X \left(\phi'\right)^{-1}(\alpha+\lambda \cdot a(t))\,d\mu(t)$ is a continuous function of $\alpha$ by applying the intermediate value theorem. Uniqueness follows from the strict monotonicity of $\left(\phi'\right)^{-1}$. The regularity follows from the implicit function theorem, using an argument analogous to Proposition \ref{propSingularPotSmooth}. The existence of a maximiser for the problem in Equation \eqref{eqMax1} follows from the existence of a unique critical point and the strict concavity of the functional. 
\end{proof}

\begin{remark}\label{remarkShannonConstant}
In the case of the Shannon entropy, the additive property of the exponential greatly simplifies the previous lemma, since $\alpha_\lambda$ can explicitly be given as 
\begin{equation}
\alpha_\lambda = 1-\ln\left(\int_X \exp\left(\lambda \cdot a(t)\right)\,d\mu(t)\right).
\end{equation}
\end{remark}

We now define $\lambda \in\mathbb{R}^k$ to be the dual order parameter. It corresponds to the classical order parameter through the relationship 
\begin{equation}
\vec{b}=\int_X \left(\phi'\right)^{-1}\left(\alpha_\lambda +\lambda \cdot a(t)\right)a(t)\,d\mu(t).
\end{equation}
If $\phi$ is $C^m$ for $m\geq 2$, then it follows that the relationship between the dual and classical order parameters is a $C^{m-1}$ bijection between $\mathbb{R}^k$ and $\mathcal{Q}$. Furthermore, if $\vec{b}^\lambda\in\mathcal{Q}$ is the classical order parameter corresponding to a dual order parameter $\lambda$, then it is immediate that 
\begin{equation}\mathcal{I}_{\mathcal{Q}}(\vec{b}^\lambda)=\alpha_\lambda + \lambda \cdot \vec{b}^\lambda + f(\vec{b}^\lambda)=\mathcal{I}_{\mathbb{R}^k}(\lambda)\end{equation}
The map $\lambda \mapsto \vec{b}^\lambda$ is a continuous open map, which implies that local minimisers of $\mathcal{I}_{\mathbb{R}^k}$ and $\mathcal{I}_{\mathcal{Q}}$ are in one-to-one correspondence also. The advantage of solving $\mathcal{I}_{\mathbb{R}^k}$ is that the domain is no longer a constrained set, and at each point one must only solve a 1-dimensional optimisation problem rather than a $(k+1)$-dimensional one. In the case of the Shannon entropy, the explicit expression of $\alpha_\lambda$ as mentioned in Remark \ref{remarkShannonConstant} requires no optimisation problem to be solved in order to evaluate the energy. In the case of the Onsager energy, the following proposition explicitly states these conclusions.

\begin{proposition}
Let $\lambda \in\mathbb{R}^k$. Define 
\begin{equation}
\begin{split}
Z_\lambda =& \int_X \exp(\lambda \cdot \vec{a}(t))\,d\mu(t),\\
\vec{b}^\lambda=&\frac{1}{Z_\lambda} \int_X \exp(\lambda \cdot \vec{a}(t))\vec{a}(t)\,d\mu(t).
\end{split}
\end{equation}
Then all global minimisers, $L^1$-local minimisers and critical points (respectively) of Onsager's free energy 
\begin{equation}\mathcal{I}_{\mathcal{P}(X)}(\rho)=T\int_X \rho(t)\ln\rho(t)\,d\mu(t)-\frac{1}{2}\tens{K}\vec{b}_\rho \cdot \vec{b}_\rho \end{equation}
are in one-to-one correspondence with global minimisers, local minimisers and critical points (respectively) of the function 
\begin{equation}
T\left(\lambda \cdot \vec{b}^\lambda - \ln(Z_\lambda)\right)-\frac{1}{2}\tens{K}\vec{b}^\lambda \cdot \vec{b}^\lambda
\end{equation}
by the relation
\begin{equation}
\rho=\frac{1}{Z_\lambda}\exp(\lambda \cdot \vec{a}(t)).\end{equation}
\end{proposition}
\begin{proof}
This follows from the preceding discussion. 
\end{proof}

Finally using the machinery of the singular potential it is possible to obtain some simple estimates on the stability of the isotropic phase $\rho(t)=\frac{1}{\mu(X)}$ in mean field free energy models. In the following $\phi$ will be taken as the Shannon entropy $\phi(x)=x\ln x$, and the constraint functions $a_i$ are taken to be orthonormal in $L^2(X)$. Define the mean field free energy on $\mathcal{P}(X)$ for $T>0$ and a positive definite $k \times k$ matrix $\tens{K}$ as
\begin{equation}
\mathcal{I}_{\mathcal{P}}(\rho)=T\int_X \rho \ln \rho\,d\mu - \frac{1}{2}\tens{K}\vec{b}\cdot \vec{b}
\end{equation}
\begin{proposition}\label{propLocalStability}
If $T>||\vec{a}||_\infty^2\lambda_{\max}(\tens{K})$ then the isotropic state is globally stable and there are no other critical points of $\mathcal{I}_{\mathcal{P}}$. If $T>\frac{1}{\mu(X)}\lambda_{\max}(\tens{K})$ then the isotropic state is at least locally stable. If $T< \frac{1}{\mu(X)}\lambda_{\max}(\tens{K})$ then the isotropic state is unstable.
\end{proposition}
\begin{proof}
Consider the corresponding free energy functional $\mathcal{I}_\mathcal{Q}:\mathcal{Q} \to \mathbb{R}$ defined by 
$$\mathcal{I}_{\mathcal{Q}}(\vec{b})=T\psi_s(\vec{b})-\frac{1}{2}\tens{K}\vec{b}\cdot \vec{b}$$
Note first that the isotropic state $b=0$ is always a critical point for $\mathcal{I}_{\mathcal{Q}}$. The Hessian matrix of $\mathcal{I}_{\mathcal{Q}}$ at $\vec{b}$ is given by 
\begin{equation}
\tens{H}(\vec{b})=\frac{\partial^2 \mathcal{I}_{\mathcal{Q}}}{\partial \vec{b}^2}(\vec{b})=T\left(\int_X \vec{a} \otimes \vec{a} \rho_{\vec{b}} - \vec{b} \otimes \vec{b} \right)^{-1} - \tens{K}.
\end{equation}
The H\"older inequality and that $\vec{b}\otimes \vec{b}$ is positive semi-definite gives
\begin{equation}
\int_X \vec{a} \otimes \vec{a} \rho_{\vec{b}}- \vec{b} \otimes \vec{b}\leq \int_X \vec{a} \otimes \vec{a} \rho_{\vec{b}} \leq ||\vec{a}||_\infty^2 \tens{I}
\end{equation}
so that $\tens{H}(b)\geq \frac{T}{||a||_\infty^2}\tens{I-K}$. This gives that for $T>||a||_\infty^2 \lambda_{\max}(\tens{K})$ that $\tens{H}(\vec{b})$ is positive definite for all $\vec{b} \in \mathcal{Q}$ so that $\mathcal{I}_{\mathcal{Q}}$ is a strictly convex function on $\mathcal{Q}$, so that there exists at most one critical point. This implies that the isotropic state must be the global minimum and only critical point. Applying Proposition \ref{propGlobalMinsEquiv} gives the corresponding result for $\mathcal{I}_{\mathcal{P}}$. Due the the normalisation of the functions $(a_i)_{i=1}^k$ it holds that $\tens{H}(0)=T\mu(X)\tens{I-K}$. Consequently if $T>\frac{1}{\mu(X)}\lambda_{\max}(\tens{K})$ then $\tens{H}(\vec{0})$ is positive definite so that $\vec{b}=\vec{0}$ is a local minimum. Conversely, if $T<\frac{1}{\mu(X)}\lambda_{\max}(\tens{K})$ then $\tens{H}(\vec{0})$ has a negative eigenvalue so that $\vec{b}=\vec{0}$ is unstable. Applying Proposition \ref{propLocalMins} gives the desired result for the functional $\mathcal{I}_{\mathcal{P}}$.\qed
\end{proof}

\begin{remark}
If $L^\infty$-local minimisers are {\it a priori} known to be bounded away from zero and $+\infty$, then $L^\infty$ variations can obtain a critical point condition for $L^\infty$-local minimisers analogous to Corollaries \ref{corollaryCriticalPointCondition1} and \ref{corollaryCriticalPointCondition2}. Similarly, taking the second variation can provide estimates analogous to Proposition \ref{propLocalStability} to assess the stability under $L^\infty$ perturbations. The advantages of the methods presented in this work are twofold; firstly that all statements are in reference to $L^1$-local minimisers, a stronger condition which could not be obtained by two-sided variations since $\mathcal{P}(X)$ has empty relative interior in $L^1(X)$, and secondly the technicality that local minimisers must be shown to be bounded away from zero and $+\infty$ is removed. 
\end{remark}

\begin{remark}
\label{remarkMeanFieldInfinite}
More general mean-field like free energies can be defined by 
\begin{equation}
\mathcal{I}_{\mathcal{P}}(\rho)=\int_X \rho \ln \rho \, d\mu - \int_X \int_X K(t,s)\rho(t)\rho(s)\,d\mu(t)\,d\mu(s)
\end{equation}
with $K\in L^\infty(X \times X)$ a symmetric kernel. If, for example, $X$ is a connected, analytic, bounded, Riemannian manifold then the eigenvectors of the Laplace-Beltrami operator on $X$, denoted $(a_i)_{i=1}^\infty$, form a countable, dense, orthogonal basis for $L^2(X)$ with $a_i$ analytic for $i\in \mathbb{N}$ \cite{berard1986spectral}. By Prop. \ref{propPseudoHaarManifold} they must therefore form a pseudo-Haar set. This allows the decomposition of $K$ as 
\begin{equation}
\label{eqDecompose} K(t,s)=\sum\limits_{i,j=1}^\infty c_{ij}a_i(t)a_j(s),
\end{equation}
so that by truncating this series to finitely many terms the kernel can be approximated, and the approximating model fits into the framework presented in this paper. In many models related to liquid crystals this decomposition would be superfluous, as the kernel is assumed to have a decomposition of the form in \eqref{eqDecompose} with $c_{ij}$ non-zero for only finitely many $i,j$. For example, if $V_\lambda(X)$ denotes the eigenspace of the Laplace-Beltrami operator on $X$ with corresponding eigenvalue $\lambda$, then the constraint functions corresponding to the Maier-Saupe potential (dipolar potential, respectively) for nematic liquid crystals are in $V_{-6}(\mathbb{S}^2)$ ($V_{-2}(\mathbb{S}^2)$, respectively) \cite{fatkullin2005critical}, the McMillan model (see Subsubsection \ref{subsubsecMcMillan}) has two constraint functions, one in $V_{-6}(\mathbb{S}^2\times[0,1])$ and another in $V_{2\pi}(\mathbb{S}^2\times[0,1])$, and the Strayley model \cite{straley1974ordered} uses elements of $V_{-6}(\textup{SO}(3))$. 
\end{remark}

\begin{remark}
Propositions  \ref{propGlobalMinsEquiv}, \ref{propLocalMins} and \ref{propCriticalPointsEquivalent} presented together imply that globally stable, locally stable and unstable equilibrium points of the functionals on $\mathcal{P}(X)$ can be found by considering a macroscopic functional defined only on $\mathcal{Q}$, combined with the maximum entropy assumption.
\end{remark}

\subsection{Models with Spatial Inhomogeneities}
\label{subsecInhomo}

If, as in the Q-tensor theory of liquid crystals, one wishes to consider models with inhomogeneities in a domain $\Omega \subset \mathbb{R}^n$, then one could appeal to the calculus of variations and consider functions $b:\Omega \to \overline{\mathcal{Q}}$ (note that the closure of $\mathcal{Q}$ is taken for compactness properties), and investigate minimisers of the functional 
\begin{equation}
\mathcal{I}(\vec{b})=\int_\Omega W(\nabla \vec{b}(\vec{x}),\vec{b}(\vec{x})) + F(\vec{b}(\vec{x}))\,d\vec{x}
\end{equation}
in some appropriate function space, where $W:\mathbb{R}^{k\times n}\times \mathcal{Q}\to\mathbb{R}$ represents the energy of distortions, and $F:\overline{\mathcal{Q}} \to\mathbb{R}$ represents the free energy of a homogeneous system. However, if $F$ were to be bounded on $\partial\mathcal{Q}$, as is the case for a Landau expansion, then it is possible that minimisers could satisfy $\vec{b}(x)\in\partial\mathcal{Q}$ for all $\vec{x}$ in some set of positive measure, providing unphysical solutions. Furthermore this is undesirable from a mathematical perspective since the minimiser would not satisfy the Euler-Lagrange equation. However, if one considers a functional $\mathcal{I}_s$, given by 
\begin{equation}
\mathcal{I}_s(\vec{b})=\int_\Omega W(\nabla \vec{b}(\vec{x}),\vec{b}(\vec{x})) + \psi_s(\vec{b}(\vec{x}))+f(\vec{b}(\vec{x}))\,dx
\end{equation}
for some $f,W$ bounded away from $-\infty$, then provided finite energy configurations exist, it is immediate that minimisers satisfy $\vec{b}(\vec{x}) \in \mathcal{Q}$ for almost every $\vec{x}\in\Omega$. Furthermore, in particular simple models it is possible to show strict physicality, in the sense that there exists some compact set $K \subset \mathcal{Q}$ for which any global minimiser $\vec{b}$ satisfies $\vec{b}\in K$ almost everywhere, which has the consequence that global minimisers satisfy the Euler-Lagrange equation allowing PDE methods to be applied to the problem. This is shown by a projection method based similar to that in Ref. \refcite{ball2014equilibrium}.

\begin{lemma}
\label{lemmaPhysicality}
Let $g:\mathcal{Q}\to\mathbb{R}$ be Lipschitz continuous with Lipschitz constant $L$. Given $M\in \mathbb{R}$ define the set $K_M=\{\vec{b} \in \mathcal{Q}:\psi_s(\vec{b})\leq M\}$, and the nearest point projection $P_M:\mathcal{Q} \to K_M$. Furthermore given $\vec{b} \in \mathcal{Q}$ denote $\vec{b}_M = P_M\vec{b}$. Then there exists $M_0\in \mathbb{R}$ so that for all $M>M_0$ and $\vec{b} \in \mathcal{Q}$, \begin{equation}
\psi_s(\vec{b})+g(\vec{b})\geq \psi_s(\vec{b}_M)+g(\vec{b}_M).
\end{equation}
\end{lemma}
\begin{proof}
First note that the projection is well defined, since the convexity and continuity of $\psi_s$ ensure that $K_M$ is a closed convex set. Suppose that the result is false. Then there must exist sequences $M^k\in\mathbb{R}$, $\vec{b}_k\in\mathcal{Q}$ so that $\lim\limits_{k \to \infty}M^k=+\infty$, and defining $\vec{c}_k=P_{M^k}\vec{b}_k$, for each $k$ the inequality
$$\psi_s(\vec{c}_k)+g(\vec{c}_k)>\psi_s(\vec{b}_k)+g(\vec{b_k})$$
is satisfied. By taking a subsequence if necessary (not relabelled), assume that $(M^k)_{k\in\mathbb{N}}$ is an increasing sequence. Using the fact that $g$ is Lipschitz and elementary inequalities for convex functions, it holds that
\begin{equation}
\begin{split}
L|\vec{c}_k-\vec{b}_k| & \geq g(\vec{c}_k)-g(\vec{b}_k)\\
&> \psi_s(\vec{b}_k)-\psi_s(\vec{c}_k)\\
&\geq \nabla \psi_s(\vec{c}_k)\cdot (\vec{b}_k - \vec{c}_k).
\end{split}
\end{equation}
From the construction of $K_M$, $\partial K_M$ is a level set of $\psi_s$, so that $\nabla\psi_s(\vec{c}_k)$ is normal to the surface $\partial K_{M^k}$ at $\vec{c}_k$. Also from the properties of projections it holds that $\vec{b}_k - \vec{c}_k$ is normal to $\partial K_{M^k}$ at $\vec{c}_k$, with the same sign as $\nabla \psi_s(\vec{c}_k)$ in the sense that $\nabla\psi_s(\vec{c}_k) = s(\vec{b}_k-\vec{c}_k)$ for some $s>0$. This implies $\nabla \psi_s(\vec{c}_k)\cdot(\vec{b}_k-\vec{c}_k)=|\nabla\psi_s(\vec{c}_k)|\,|\vec{b}_k-\vec{c}_k|$. Combining this with the previous chain of inequalities gives 
\begin{equation}
\begin{split}
L|\vec{c}_k-\vec{b}_k| & \geq |\nabla \psi_s(\vec{c}_k)| \, |\vec{c}_k - \vec{b}_k|\\
\Rightarrow L & \geq |\nabla \psi_s (\vec{c}_k)|
\end{split}
\end{equation}
for all $k$. However since $\mathcal{Q}$ is bounded, a subsequence (not relabelled) can be taken so that $\vec{c}_k \to \vec{c}^*$, and since $\vec{c}_k\in \partial K_{M^k}$ for all $k$, this gives that $\vec{c}^* \in \partial \mathcal{Q}$. Since $\lim\limits_{\vec{b}\to \partial \mathcal{Q}}\psi_s(\vec{b}) = +\infty$, and $\psi_s$ is convex, this also implies that $\lim\limits_{\vec{b} \to \partial \mathcal{Q}}|\nabla \psi_s(\vec{b})|=+\infty$. Taking the limit as $k\to\infty$ of the inequality $L \geq |\nabla \psi_s(\vec{c}_k)|$ gives a contradiction, completing the proof.\qed
\end{proof}

\begin{theorem}
Let $\Omega \subset \mathbb{R}^n$. Consider the minimisation problem
\begin{equation}
\min\limits_{\vec{b} \in \mathcal{A}}\int_\Omega |\nabla \vec{b}(\vec{x})|^p + \psi_s(\vec{b}(\vec{x}))+g(\vec{b}(\vec{x}))\,dx,
\end{equation}
where $\mathcal{A}=\{\vec{b} \in W^{1,p}(\Omega;\overline{\mathcal{Q}}): \vec{b}|_{\partial\Omega}=\vec{b}_0\}\neq \emptyset$, $g:\overline{\mathcal{Q}}\to\mathbb{R}$ is Lipschitz continuous with Lipschitz constant $L$, $p>1$ and there exists some compact $K \subset \mathcal{Q}$ such that $\vec{b}_0(x)\in K$ for $\mathcal{H}^{n-1}$-almost every $x \in \partial \Omega$. Then there exists a global minimiser, and all global minimisers are strictly physical in the sense that $\psi_s(\vec{b})\in L^\infty(\Omega,\mathbb{R})$. In particular, any global minimiser $\vec{b}^*:\Omega \to \overline{\mathcal{Q}}$ satisfies the Euler-Lagrange equation for the energy functional if $g\in C^1(\mathcal{Q})$.
\end{theorem}
\begin{proof}
The existence of minimisers follows from a standard direct method argument \cite{dacorogna2007direct}. To see that the global minimiser must be strictly physical, assume for the sake of contradiction that there exists a global minimiser $\vec{b} \in \mathcal{A}$ such that $\psi_s(\vec{b})$ is unbounded on $\Omega$. Let $M>M_0$ (as defined in Lemma \ref{lemmaPhysicality}) and $\vec{b}_M$ be the projection of $\vec{b}$ onto the set $K_M$ as defined in Lemma \ref{lemmaPhysicality}. Take $M$ sufficiently large so that $K \subset K_M$, so that $\vec{b}_M$ satisfies the same boundary conditions as $\vec{b}$. Then $\psi_s(\vec{b}_M)+g(\vec{b}_M) \leq \psi_s(\vec{b})+g(\vec{b})$ almost everywhere in $\Omega$ by Lemma \ref{lemmaPhysicality}. Furthermore, since the term on the left is bounded, and by assumption the term on the right is unbounded this equality is strict on a set of positive measure. Finally since $|\nabla P\vec{b}| \leq |\nabla \vec{b}|$ for any function in $W^{1,p}$ and projection $P$ onto a convex set (see Lemma 10 in Ref. \refcite{ball2014equilibrium}), it holds that 
\begin{equation}
\int_\Omega |\nabla \vec{b}_M(\vec{x})|^p + \psi_s(\vec{b}_M(\vec{x}))+g(\vec{b}_M(\vec{x}))\,d\vec{x} < \int_\Omega |\nabla \vec{b}(\vec{x})|^p + \psi_s(\vec{b}(\vec{x}))+g(\vec{b}(\vec{x}))\,d\vec{x},
\end{equation}
contradicting that $\vec{b}$ is a global minimiser. Since the minimiser is bounded away from $\partial \mathcal{Q}$ it is possible to take smooth variations so that the minimiser satisfies the Euler-Lagrange equation. \qed
\end{proof}

\subsection{Approximation of the Singular Potential by Everywhere Defined Functions}
In some situations it may be preferable to approximate the singular potential by a globally defined function. In particular, as mentioned in Subsection \ref{subsecInhomo}, in models with inhomogeneities the blow up property can pose issues with minimisers satisfying the Euler-Lagrange equation. One tool kit for providing such an approximation is to use the Landau theory \cite{toledano1987landau}, which argues that the free energy of a system must be analytic in the order parameters, therefore it is possible to replace the free energy with a polynomial in the order parameters that respects the symmetry of the system. Typically, this polynomial will be taken to fourth order. However, as will be shown in this section, the Landau expansion is not necessarily compatible with free energies of the form given in Subsection \ref{SubsecMeanField} , since the Taylor approximation does not necessarily preserve shape properties possessed by the singular potential. This section will only be concerned with the Shannon entropy, due to the relative simplicity of its Taylor approximation as well as its physical relevance.

\subsubsection{One Dimensional Examples and Counterexamples}
The fourth order Taylor approximation about zero of the singular potential is derived in Appendix \ref{appendixTaylorExpand}. For simplicity consider only a single constraint function, and state space $X=[-1,1]$. Define $m_i=\frac{1}{2}\int_X a(x)^i\,dx$. The the fourth order Taylor approximation $\psi_s^4:\mathbb{R}\to\mathbb{R}$ is given by 
\begin{displaymath}
\begin{split}
\psi_s^4(b)=&\psi_s(0)+\sum\limits_{j=2}^4\frac{d^j \psi_s}{db^j}(0) \frac{b^j}{j!}\\
=&\frac{1}{2m_2}b^2 -\frac{m_3}{6m_2^3}b^3+\frac{3m_3^2-m_2m_4+3m_2^3}{24m_2^5}b^4+\psi_s(0).
\end{split}
\end{displaymath}
If the Taylor approximation is to represent a kind of macroscopic version of entropy, and approximate the singular potential, then at the very least it should be convex, posses a single local minimum, and be coercive so that the approximation blows up to $+\infty$ as $b \to \pm \infty$. Coercivity perhaps the most important property, since if it fails then energy minimisation is not possible. The second condition prevents the appearance of ``phantom" minimisers for the entropy which should not exist. Finally convexity is a useful property in minimisation problems that one would want to inherit from the singular potential. In the one dimensional case presented, these three conditions have a chain of implication, so that convexity implies the existence of a single critical point, and a single critical point implies coercivity. Note that these implications rely on the fact that $\psi_s^4$ is convex in some neighbourhood of the origin, and that it is not linear. Similarly since $\psi_s^4$ is a polynomial that is not affine, convexity implies strict convexity.

Fortunately in the one dimensional case, verifying if these properties hold for a given example is relatively straightforward. In order to be coercive, the coefficient of $b^4$ must be positive, which gives the condition 
$$d_1=3m_3^2-m_2m_4-3m_2^3>0.$$
If $b\neq 0$ is a critical point of $\psi_s^4$, this implies the existence of a non-zero real solution to 
$$\frac{d\psi_s^4}{db}(b)=b\left(\frac{1}{m_2}-\frac{1}{2m_2^3}m_3b+\frac{1}{6m_2^5}\left(3m_2^3-m_2m_4+3m_3^2\right)b^2\right)=0.$$
By considering the discriminant of the quadratic, this means that $\psi_s^4$ possesses no non-trivial critical points if and only if 
$$d_2=72m_2^3-24m_2m_4+63m_3^2>0.$$
Finally for convexity, the second derivative of $\psi_s^4$ is a quadratic in $b$ given by 
$$\frac{d^2\psi_s^4}{db^2}(b)=\frac{1}{m_2}-\frac{m_3}{m_2^3}b+\frac{1}{2m_2^5}\left(3m_2^3-m_2m_4+3m_3^2\right)b^2,$$
so by the same argument $\psi_s^4$ is convex if and only if 
$$d_3=6m_2^3-2m_2m_4+5m_3^2>0.$$
Within the chain of implication that
$$\text{Convexity}\Rightarrow \text{A single critical point} \Rightarrow \text{Coercivity},$$ 
none of the implications are equivalent for the singular potential. Using the domain $X=[-1,1]$, there are four examples that produces all four possibilities.

\begin{figure}[H]
\begin{center}
\begin{tabular}{|m{3cm}|m{0.5cm}|m{3cm}|m{1.9cm}|m{1.3cm}|@{}m{0cm}@{}}
\hline
 $a(x)=$ &	$x$ 	&	 $x^4-7x^3-2x^2+3x+\frac{7}{15}$	&	$ 7x^3-x^2+x+\frac{1}{3}$	&	 $x^3$&	\\[1.5ex] \hline
	$m_2=$ 	&	$\frac{1}{3}$	&	$\frac{904}{525}$	&	$\frac{92}{9}$	&	$\frac{1}{7}$&	\\[1.2ex] \hline
	$m_3=$	&	$0$	&	$\frac{-6796096}{3378375}$	&	 $\frac{-12112}{945}$	&	$0$	&\\[1.2ex] \hline
	$m_4=$	&	$\frac{1}{5} $	&	$\frac{4297061248}{287161875}$	&	 $\frac{6938192}{19305}$	&	$\frac{1}{13}$	&\\[1.2ex] \hline
	$d_1=$	&	$\frac{2}{45}$	&	$\frac{765059068785152}{452732233078125}$	&	$\frac{998000512}{42567525}$	&	$\frac{-10}{4459}$	&\\[1.2ex] \hline
	Coercive 	&Yes	&Yes		&	Yes 	&No		&\\[1.2ex] \hline
	$d_2=$	&	$\frac{16}{15} $	&	$\frac{624228970176512}{150910744359375}$	&	$\frac{-618770176}{675675}$	&	$\frac{-240}{4459}$	&\\[1.2ex] \hline
	Single Critical Point	&	Yes&Yes		&No	 	&No		&\\[1.2ex]\hline
	$d_3=$	&	$\frac{4}{45}$	&	$\frac{-905889167393792}{1358196699234375}$	&	$\frac{-2998034944}{25540515}$	&	$\frac{-20}{4459}$	&\\[1.2ex] \hline
	Convex	&	Yes&No		&	No 	&No		&\\[1.2ex]\hline
\end{tabular}
\caption{Table of examples and discriminant values}
\end{center}
\end{figure}

As these simple examples show, the Taylor expansion approach to providing a polynomial approximation to the free energy will not necessarily work in general, and this is perhaps unsurprising given that the Taylor approximation is generally only a local approximation of a function. However, the following Weierstrass-type result shows that convex polynomial approximation is possible.

\begin{proposition}[From  Ref. \refcite{shvedov1981coconvex}]
Let $M \subset \mathbb{R}^k$ be convex and compact, and $f:M \to \mathbb{R}$ be convex and continuous. Then for every $\epsilon >0$ there exists some polynomial $p_\epsilon$ which is convex on $\mathbb{R}^k$ such that 
$$\sup\limits_{\vec{x} \in M}|f(\vec{x})-p_\epsilon(\vec{x})|<\epsilon.$$
\end{proposition}

The main limitations of this result however are that the approximation can only be performed on compact subsets of $\mathcal{Q}$, and also, much like the classical Weierstrass result, the proof is non-constructive. It should also be noted that this result requires the function to be only continuous, rather than analytic, so that the Landau theory's assumption of an analytic free energy is unnecessary. One might hope that it is possible to establish a necessary and sufficient condition for the fourth order Taylor approximation to be convex, but in general the problem of establishing if a given polynomial is convex is an NP-hard problem \cite{ahmadi2013np}. Rather than pursuing a polynomial approximation, the Yosida-Moreau regularisation may be more appropriate, due to its shape preservation properties.

\begin{definition}
For $J>0$, define the Yosida-Moreau regularisation of $\psi_s$, denoted $\psi^J$ by 
\begin{equation}
\label{eqYMdef}
\psi^J(\vec{b})=\min\limits_{\tilde{\vec{b}} \in \mathcal{Q}}\psi_s(\vec{b})+\frac{J}{2}|\tilde{\vec{b}}-\vec{b}|^2.
\end{equation}
\end{definition}

The following proposition outlines several of the key properties which suggest that the Yosida-Moreau approximation is an appropriate approximation for the singular potential. The results presented are for the general entropy-like objective functions and pseudo-Haar constraint functions. 

\begin{proposition}
\begin{enumerate}
\item \label{itemMins} The Yosida-Moreau approximation preserves minima, in the sense that for all $J>0$, $\min\psi^J=\psi^J(\vec{0})=\psi_s(\vec{0})=\min\psi_s$. For all $\vec{b} \in \mathbb{R}^k$, $J>0$, the minimisation problem defining $\psi^J$, as defined in Equation \eqref{eqYMdef}, admits a unique solution denoted $G_J(\vec{b})$.
\item \label{itemConv} For every $\vec{b} \in \mathbb{R}^k$, $\psi^J(\vec{b})\nearrow \psi_s(\vec{b})$, where the limit is infinite for $\vec{b} \not\in\mathcal{Q}$. 
\item \label{itemRegular} For all $J>0$, $\psi^J$ is a convex, differentiable function. Furthermore, the gradient is Lipschitz with Lipschitz constant $J$. The derivative of $\psi^J$ can be given in terms of $G_J(\vec{b})$ by 
\begin{equation}
\frac{\partial \psi^J}{\partial b}(\vec{b})=J\left(\vec{b}-G_J(\vec{b})\right)=\lambda\left((G_J(\vec{b})\right).
\end{equation}
\item \label{itemDual} The Yosida-Moreau approximation can be evaluated numerically by the dual problem,
\begin{equation}
\psi^J(\vec{b})=\max\limits_{\alpha \in \mathbb{R},\lambda \in \mathbb{R}^k} \alpha + \lambda \cdot \vec{b} -\int_X \phi^*(\alpha+\lambda\cdot a)\,d\mu -\frac{1}{2J}|\lambda|^2.
\end{equation}
\end{enumerate}
\end{proposition}
\begin{proof}
The preservation of minima in statement \ref{itemMins} is immediate from the definition, by testing $\tilde{\vec{b}}=\vec{0}$. The existence and uniqueness of a minimiser follow from the coercivity and strict convexity of the objective function. Statement \ref{itemConv} can be found in Ref. \refcite{moreau1965proximite}. Statement \ref{itemRegular} can be found in Ref. \refcite{brezis1973ope}. Statement \ref{itemDual} is given in Ref. \refcite{decarreau1992dual}.\qed
\end{proof}
\begin{remark}
The previous proposition has several consequences. Firstly, statements 1 and 2 say that $\psi^J$ approximates $\psi_s$, as well as preserving desirable shape properties. Statement 3 gives that $\psi^J$ is sufficiently regular for first order methods to be used. Lastly, statement 4 shows that the dual optimisation problem defining to $\psi^J$ is, at face value, no harder to approach than the optimisation problem defining $\psi_s$.
\end{remark}

\subsubsection{The McMillan Model}\label{subsubsecMcMillan}
As an illustrative and physically meaningful example consider the McMillan model for Isotropic-Nematic-Smectic-A phase transitions \cite{mcmillan1971simple}. This is a mean field theory as described previously where the state space is $X=\mathbb{S}^2\times [0,1]$, and two constraint functions are given by 
\begin{displaymath}
\begin{split}
a_1(\vec{p},x)=&\frac{1}{2}\left(3(\vec{p}\cdot \vec{e}_1)^2-1\right),\\
a_2(\vec{p},x)=&\frac{1}{2}\left(3(\vec{p}\cdot \vec{e}_1)^2-1\right)\cos(2\pi x).
\end{split}
\end{displaymath}
Here $\vec{e}_1$ is a unit vector, physically corresponding to the orientation of the material. Due to the rotational symmetry of the constraint functions, it is possible to consider only state variables $(\theta,x)\in [0,\pi]\times[0,1]$, where $\cos(\theta)=\vec{p}\cdot \vec{e}_1$. This approach views $X$ as equivalent to a subset of $\mathbb{R}^2$ with measure $d\mu(\theta,x)=2\pi\sin(\theta)\,d\theta\,dx$. The constraint functions are analytic and linearly independent, and since the measure $\mu$ has the same null sets as the Lebesgue measure, the pseudo-Haar condition is satisfied. These constraints give two order parameters, denoted
\begin{displaymath}
\begin{split}
S&=2\pi\int_{0}^1\int_{-\pi}^\pi a_1(\theta,x)\rho(\theta,x)\,d\theta\,dx\\
\sigma &=2\pi\int_{0}^1\int_{-\pi}^\pi a_2(\theta,x)\rho(\theta,x)\,d\theta\,dx.
\end{split}
\end{displaymath} Loosely speaking $S$ corresponds to the degree of orientational order of the molecules, and $\sigma$ represents a coupling between the order of the molecules and the location of their centre of mass. If $S=\sigma=0$, then the sample is in an isotropic phase. If $S\neq 0$ and $\sigma =0$ then it is a nematic phase, and if $S\neq 0$ and $\sigma \neq 0$ then it is a smectic A phase. Before any analysis can be performed, the set of physical moments will be established.
\begin{proposition}
The set $\mathcal{Q}$ for the constraint functions of the McMillan model is given by
\begin{equation}
\mathcal{Q}=\left\{(S,\sigma)\in\mathbb{R}^2:S\in\left(-\frac{1}{2},1\right),|\sigma|<\frac{S+2}{3}\right\}.
\end{equation}
\end{proposition}
\begin{proof}
To see that the candidate set contains $\mathcal{Q}$, use Proposition \ref{theoremNecSufCond} and test against $(\pm 1,0)$ and $\left(-\frac{1}{3},\pm 1\right)$. The maxima $a_1$ and $-a_1$ are  $1$ and $\frac{1}{2}$ respectively, which gives that $S \in \left(-\frac{1}{2},1\right)$. The maximum of $-\frac{1}{3}a_1(\theta,x)-a_2(\theta,x)$ can be found by noting that 
\begin{equation}
\begin{split}
-\frac{1}{3}a_1(\theta,x)-a_2(\theta,x) =& \frac{1}{2}\left(1-3\cos(\theta)^2\right)\left(\frac{1}{3}+\cos(2\pi x)\right)\\
\leq &\frac{1}{2} \cdot\frac{4}{3} =\frac{2}{3},
\end{split}
\end{equation}
which is attained at $\theta=\pi$ and $x=0$. For $u=\left(-\frac{1}{3},1\right)$ the same argument is used. To show that $\mathcal{Q}$ contains the candidate set is equivalent to showing that the closure of the candidate set is a subset of the closure of $\mathcal{Q}$ since the sets are convex. To prove this, it is then sufficient to show that the four vertices ($(1,\pm 1),\left(-\frac{1}{2},\pm\frac{1}{2}\right)$) are contained in the closure of $\mathcal{Q}$. This can be done constructively, and here only one vertex will be proven with the rest being shown by the same method. Let $0<\epsilon <1$, and define the set $A_\epsilon =\{(\theta,x):\cos(\theta)^2>1-\epsilon,\cos(2\pi x)>1-\epsilon\}$. Define \begin{equation}
\rho_\epsilon(\theta,x) = \left(2\pi\int_{A_\epsilon}\sin(\Theta)\,d\Theta\right)^{-1}2\pi\sin(\theta) \chi_{A_\epsilon}(\theta,x).
\end{equation}
This corresponds to a distribution uniform with respect to the measure on $X$ on $A_\epsilon$. Let the corresponding moments be denoted $S_\epsilon,\sigma_\epsilon$. Then it is immediate that 
\begin{equation}
\begin{split}
1>S_\epsilon >&\frac{1}{2}(3(1-\epsilon)-1)=1-\frac{3}{2}\epsilon\\
1>\sigma_\epsilon >& \frac{1}{2}(3(1-\epsilon)-1)(1-\epsilon)=1-\frac{5}{2}\epsilon+\frac{3}{2}\epsilon^2.
\end{split}
\end{equation}
Therefore $(S_\epsilon,\sigma_\epsilon)\in\mathcal{Q}$, and by taking $\epsilon$ to $0$, $(1,1)\in\overline{\mathcal{Q}}$. \qed
\end{proof}

Using the formula given in Appendix \ref{appendixTaylorExpand}, the exact symbolic integration package in Maple gives the fourth order Taylor approximation to the singular potential as
$$\psi_s^4(S,\sigma)= \frac{425}{196}S^4+\frac{50}{49}\sigma^2 S^2+\frac{825}{196}\sigma^4-\frac{25}{21}S^3-\frac{50}{7}\sigma^2S+\frac{5}{2}S^2+5\sigma^2 +\psi_s(0,0).$$

\begin{proposition}
For the McMillan model, $\psi_s^4$ is coercive, so that 
$$\lim\limits_{|(S,\sigma)|\to+\infty}\psi_s^4(S,\sigma)=+\infty.$$
\end{proposition}
\begin{proof}
It is sufficient to show that the fourth order terms are coercive. This can be written as a bilinear form in $S^2$ and $\sigma^2$, since 
\begin{displaymath}
\begin{split}
\frac{425}{196}S^4+\frac{50}{49}\sigma^2 S^2+\frac{825}{196}\sigma^2 =& \left[\begin{matrix} S^2 \\ \sigma^2 \end{matrix}\right]\cdot \left[\begin{matrix}\frac{425}{196} & \frac{25}{49} \\ \frac{25}{49} & \frac{825}{196}\end{matrix}\right] \left[\begin{matrix} S^2 \\ \sigma^2 \end{matrix}\right].
\end{split}
\end{displaymath}
The eigenvalues of the matrix are given by $\frac{625}{196}\pm\frac{25}{49}\sqrt{5}$, which are both positive (evaluating at approximately $2.04$ and $4.33$), so that the leading order terms are positive.
\end{proof} 
\begin{proposition}
In the McMillan model, $\psi_s^4$ is not a convex function of $S,\sigma$, even when restricted to $\mathcal{Q}$.
\end{proposition}
\begin{proof}
It is sufficient to show that the Hessian matrix of $\psi_s^4$ has a negative eigenvalue for some $(S,\sigma)$. The Hessian matrix is readily computed as 
\begin{displaymath}
\begin{split}
H(S,\sigma)=& \frac{1}{49}\left[\begin{matrix}
 1275S^2+100\sigma^2-350 S+245 & 200S\sigma-700\sigma\\ 200S\sigma-700\sigma& 100S^2+2475\sigma^2-700S+490
\end{matrix}\right],\\
\Rightarrow H(1,0)=&\frac{1}{49}\left[\begin{matrix} 1275-350+245 & 0 \\ 0 & 100-700+490\end{matrix}\right]\\
=&\frac{1}{49}\left[\begin{matrix} 1170 & 0 \\ 0 & -110\end{matrix}\right].
\end{split}
\end{displaymath}
Therefore it has a negative eigenvalue, and $\psi_s^4$ is not convex in this case. In particular, it is not convex in a neighbourhood of $(1,0)$, since the Hessian is continuous. By noting that the intersection of any neighbourhood of $(1,0)$ with $\mathcal{Q}$ is non-empty since $(1,0)\in\partial\mathcal{Q}$, this implies that $\psi_s^4$ is not convex on $\mathcal{Q}$.
\end{proof}
To conclude, contour plots of of the singular potential and its approximations are given in Figure \ref{figureContourPlots}. The singular potential itself is given in Subfigure \ref{figureContourPlotSing}. It should be noted that $\psi_s$ is only given on $\{(S,\sigma) \in \mathcal{Q}: d\left((S,\sigma),\partial\mathcal{Q}\right)>10^{-2}\}$ to avoid the difficulty in calculating the singular potential near $\partial \mathcal{Q}$. The fourth order Taylor approximation is given in Subfigure \ref{figureContourPlotTay}. Also included on the plot is a  white dashed line, which is the boundary of the set where the Hessian matrix has a negative eigenvalue. The Yosida-Moreau approximation for $J=100 $ is given inSubfigure \ref{figureContourPlotYM}. The numerical values for $\psi_s$ and $\psi^J$ were obtained via a steepest descent algorithm on the dual optimisation problem.
 
\begin{figure}[H]
\begin{subfigure}[t]{0.32\textwidth}
\includegraphics[scale=0.22]{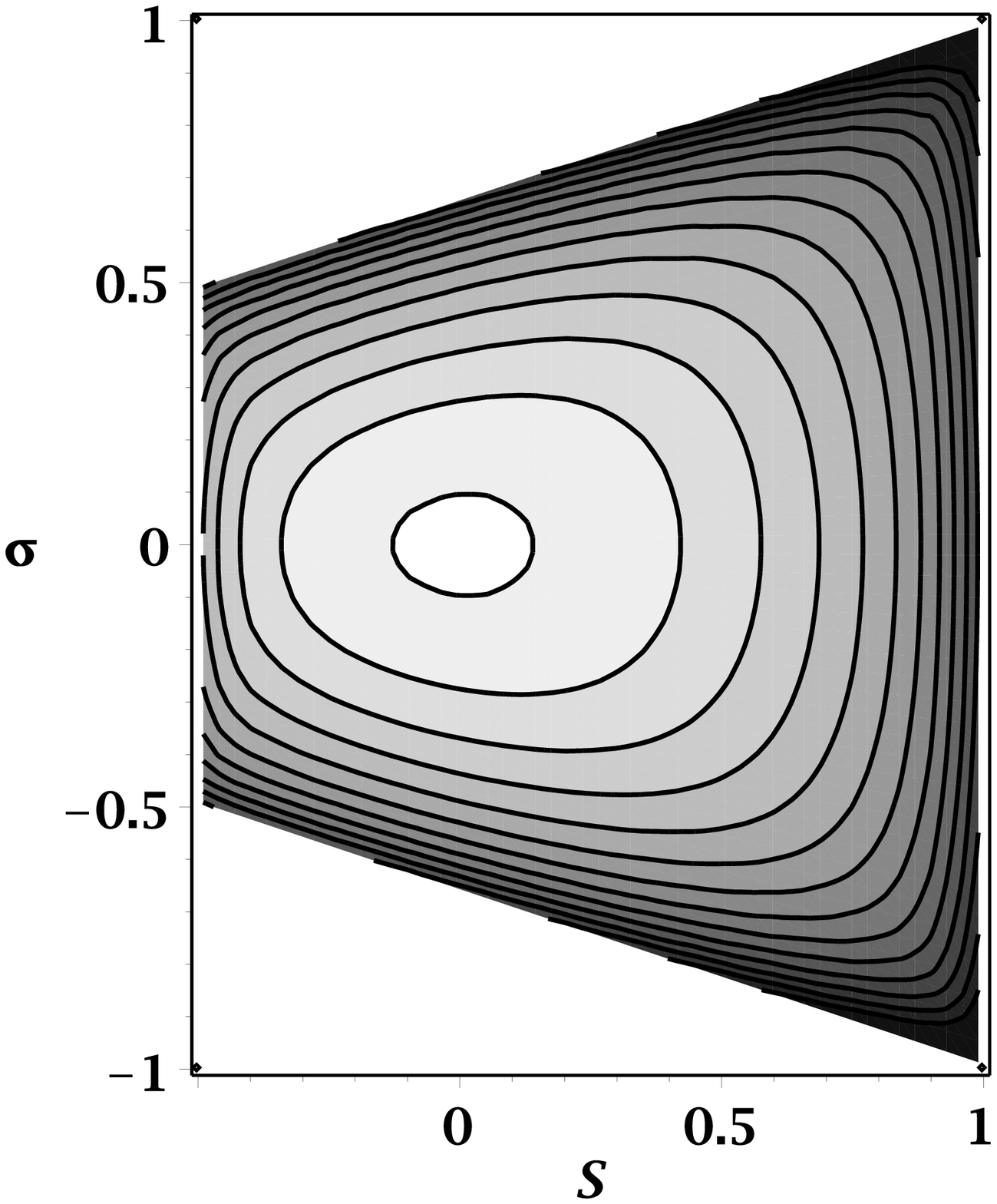} 
\caption{$\psi_s$.}
\label{figureContourPlotSing}
\end{subfigure}
\begin{subfigure}[t]{0.32\textwidth}
\includegraphics[scale=0.22]{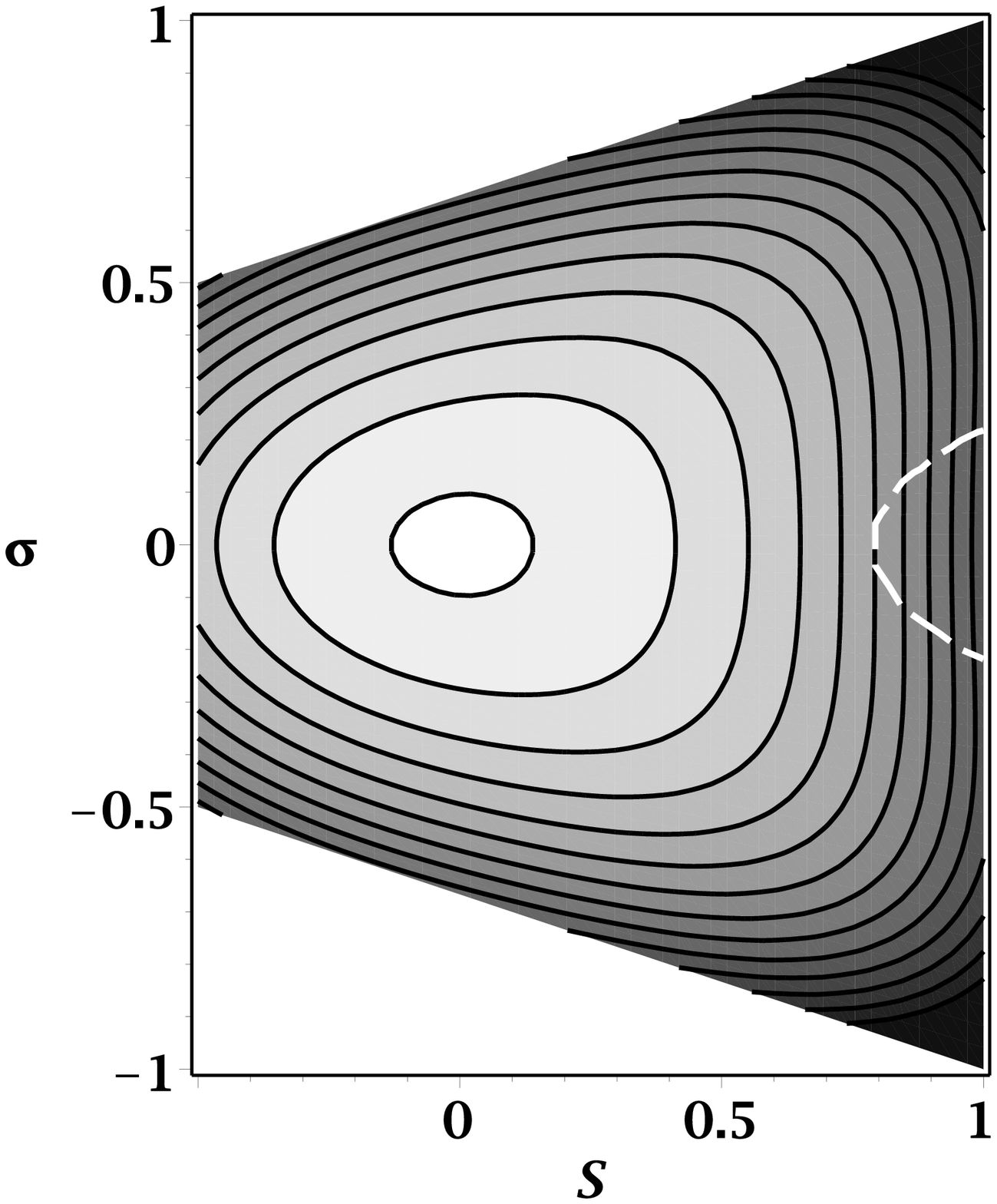} 
\caption{$\psi^4_s$.}
\label{figureContourPlotTay}
\end{subfigure}
\begin{subfigure}[t]{0.32\textwidth}
\includegraphics[scale=0.22]{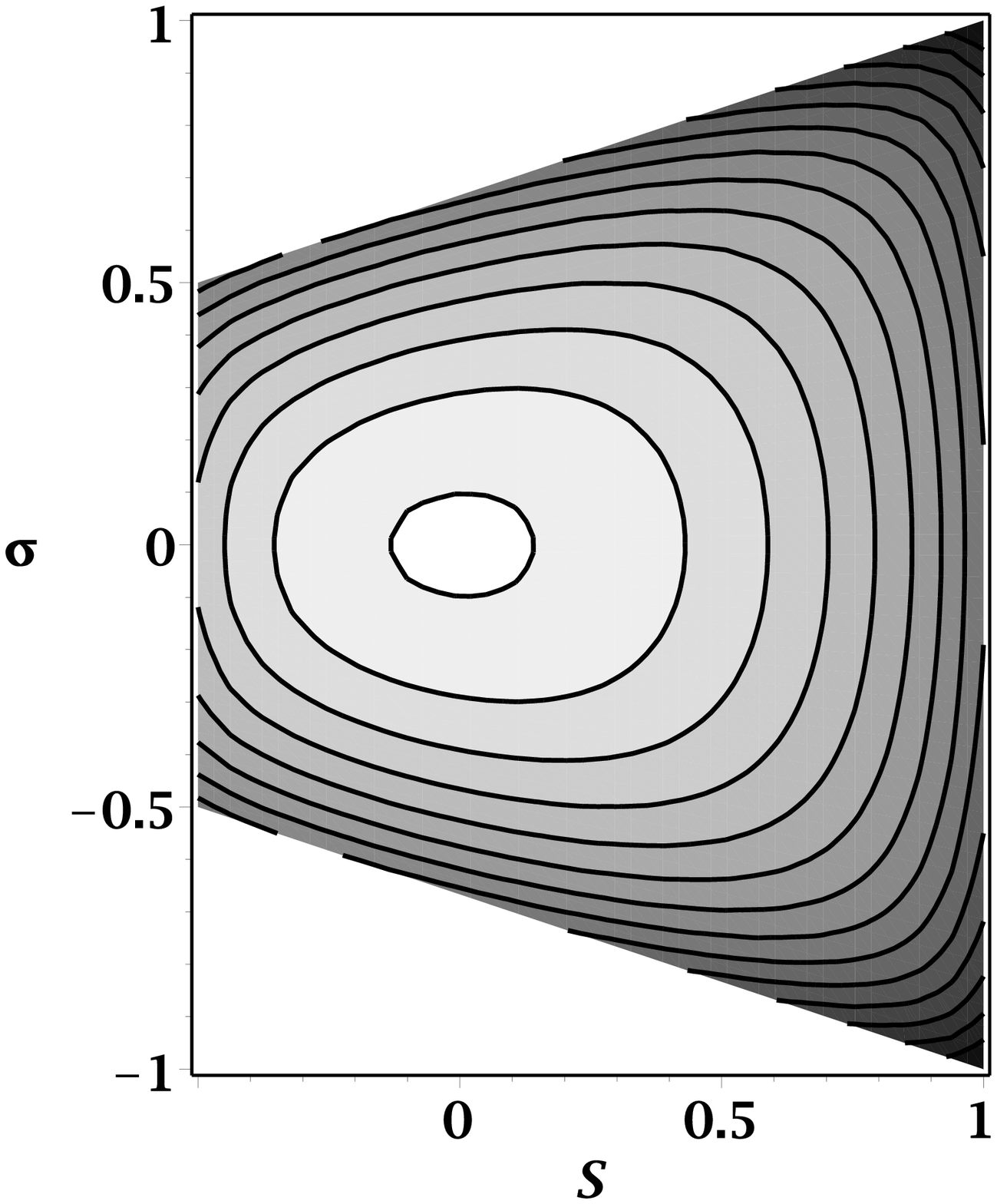} 
\caption{$\psi^J$, $J=100$.}
\label{figureContourPlotYM}
\end{subfigure}
\begin{center}
\begin{subfigure}[t]{0.5\textwidth}
\includegraphics[scale=0.3]{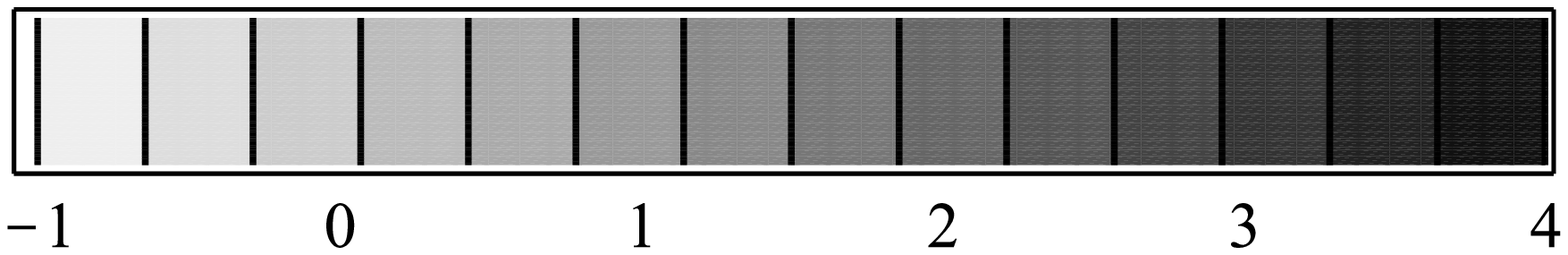}
\end{subfigure}
\end{center}
\caption{Contour plots of the singular potential and its approximations.}
\label{figureContourPlots}
\end{figure}

\appendix
\section{Appendix: The Fourth Order Taylor Expansion for Shannon Entropy}
\label{appendixTaylorExpand}
For this section consider the pseudo-Haar constraint functions $(a_i)_{i=1}^k$ to be orthogonal to the constant function so that $\int_X a_i \,d\mu=0$, but not necessarily orthogonal to each other. Furthermore the in this section the convention that Greek indices are summed over, whilst Latin indices are free, is used. Given $n \in \mathbb{N}$ and $\rho \in \mathcal{P}(X)$ define the tensor $\tens{M}^n \in (\mathbb{R}^k)^n$ component-wise by 
\begin{equation}
M^n_{i_1i_2...i_n}(\rho)=\int_X a_{i_1}a_{i_2}...a_{i_n}\rho\,d\mu.
\end{equation}
for $i_j=1,..,k$, $j=1,..n,$. Consider the function $\vec{b} \mapsto \tens{M}^n(\rho_{\vec{b}})$, which when unambiguous will simply be denoted $\tens{M}^n$, with $\rho_{\vec{b}}$ maximal entropy under the Shannon entropy $\phi(x)=x\ln x$. This allows any maximal entropy $\rho \in \mathcal{P}(X)$ to be written as 
\begin{equation}
\rho(t)=\frac{1}{Z}\exp\left(\lambda\cdot \vec{a}(t)\right).
\end{equation}
The map $\vec{b} \mapsto \tens{M}^n(\rho_{\vec{b}})$ is differentiable, so by applying the chain rule the derivative can be found as 
\begin{equation}
\begin{split}
\frac{\partial M^n_{i_1i_2...i_n}}{\partial \vec{b}_{i_{n+1}}} &= \frac{\partial M^n_{i_1i_2...i_n}}{\partial \lambda_{\alpha}}\frac{\partial \lambda_{\alpha}}{\partial b_{i_{n+1}}}\\
&= \left(\int_X a_{i_1}a_{i_2}...a_{i_n}a_{\alpha}\rho_b\,d\mu -\frac{1}{Z} M^n_{i_1i_2...i_n}\frac{\partial Z}{\partial \lambda_\alpha}\right)\left((M^2- b \otimes b)^{-1}\right)_{\alpha i_{n+1}}\\
&=\left(M^{n+1}_{i_1i_2...i_n i_\alpha}-M^n_{i_1i_2...i_n}b_\alpha\right)V_{\alpha i_{n+1}}
\end{split}
\end{equation}
where for brevity $\tens{V}=\left(\tens{M}^2-\vec{b}\otimes \vec{b}\right)^{-1}$. That the derivatives can be seen to satisfy this by applying Propositions \ref{propSingularPotSmooth} and \ref{propSingularFirstDerivative} with some elementary calculus. From here it is now possible to perform a Taylor expansion of $\psi_s$ to low order. Immediately it holds that $\psi_s(\vec{0})=-\frac{1}{\mu(X)}\ln(\mu(X))$. For the first derivative it was shown in Proposition \ref{propSingularFirstDerivative} that 
\begin{equation}
\frac{\partial \psi_s}{\partial b_i}(\vec{b})=\lambda_i.
\end{equation}
Using the strict convexity of $\phi$, it is immediate that the uniform distribution is the global entropy minimiser of $\mathcal{P}(X)$, so that in particular $\frac{1}{\mu(X)}=\rho_{\vec{b}}(t)$ for $\vec{b}=\vec{0}$ and almost every $t \in X$. By rewriting the uniform distribution on $X$ as $\frac{1}{\mu(X)}=(\phi')^{-1}(\alpha+\lambda\cdot \vec{a}(t))$ for $\lambda =\vec{0}$ and $\alpha=(\phi')\left(\mu(X)^{-1}\right)$ demonstrates that $\lambda(\vec{0})=\vec{0}$, so that 
$$\frac{\partial \psi_s}{\partial b_i}(\vec{0})=\vec{0}.$$ 
Furthermore this gives that $\vec{b}=\vec{0}$ is always a critical point for $\psi_s$ and for any truncated Taylor expansion of $\psi_s$ about $\vec{b}=\vec{0}$ with order greater than or equal to 1. For the second derivative it has been seen in Prop. \ref{propSecondDerivative} that 
\begin{equation}
\begin{split}
\frac{\partial^2 \psi_s}{\partial b_{i_1} \partial b_{i_2}}(\vec{b})&=\left(\left(\tens{M}^2(\vec{b})-\vec{b}\otimes \vec{b}\right)^{-1}\right)_{i_1i_2}\\
\Rightarrow \frac{\partial^2 \psi_s}{\partial b_{i_1}\partial b_{i_2}}(\vec{0}) &= \left(\left(M^2(\vec{0})\right)^{-1}\right)_{i_1i_2}.
\end{split}
\end{equation}
As seen before, this is necessarily a positive definite matrix, and consequently $b=0$ is always a local minimum for $\psi_s$ and any Taylor expansion of order greater than or equal to 2. For the third derivative first note that the derivative of $V$ will have to be taken and since this involves a matrix inverse a more complex expression appears,
\begin{equation}
\begin{split}
\frac{\partial V_{i_1i_2}}{\partial b_{i_3}}&= -V_{i_1\alpha_1}V_{i_2\alpha_2}\frac{\partial M^2_{\alpha_1\alpha_2}-b_{\alpha_1}b_{\alpha_2}}{\partial b_{i_3}}\\
&=-V_{\alpha_1i_1}V_{\alpha_2i_2}V_{\alpha_3i_3} \left(M^3_{\alpha_1\alpha_2\alpha_3}-b_{\alpha_1}M^2_{\alpha_2\alpha_3} - b_{\alpha_2}M^2_{\alpha_1\alpha_3} - b_{\alpha_3}M^2_{\alpha_1\alpha_2}\right).
\end{split}
\end{equation}
This gives that the third derivative at $b=0$ can be given by 
\begin{equation}
\frac{\partial^3\psi_s}{\partial b_{i_1}\partial b_{i_2}\partial b_{i_3}}(\vec{0}) = -V_{i_1\alpha_1}(\vec{0})V_{i_2\alpha_2}(\vec{0})V_{i_3\alpha_3}(\vec{0})M^3_{\alpha_1\alpha_2\alpha_3}(0).
\end{equation}
Since the Taylor expansion will only be taken to fourth order, and the expansion is to be performed around $\vec{b}=\vec{0}$, all terms with factors of $\vec{b}$ will be denoted as $A\vec{b}$, which can later be neglected in calculating the fourth derivative of $\psi_s$ at $\vec{b}=\vec{0}$. It can be given as 
\begin{equation}
\begin{split}
\frac{\partial^4 \psi_s}{\partial b_{i_1}\partial b_{i_2}\partial b_{i_3}\partial b_{i_4}} &= V_{i_1\beta_1}V_{i_2\alpha_2}V_{i_3\alpha_3}V_{i_4\beta_3} V_{\alpha_1\beta_2} M^3_{\alpha_1\alpha_2\alpha_3}M^3_{\beta_1\beta_2\beta_3}\\
&+V_{i_1\alpha_1}V_{i_2\beta_1}V_{i_3\alpha_3}V_{i_4\beta_3} V_{\alpha_2\beta_2} M^3_{\alpha_1\alpha_2\alpha_3}M^3_{\beta_1\beta_2\beta_3}\\
&+V_{i_1\alpha_1}V_{i_2\alpha_2}V_{i_3\beta_1}V_{i_4\beta_3} V_{\alpha_3\beta_2} M^3_{\alpha_1\alpha_2\alpha_3}M^3_{\beta_1\beta_2\beta_3}\\
&-\left(V_{i_1\alpha_1}V_{i_2\alpha_2}V_{i_3\alpha_3}V_{i_4\alpha_4}\right)\\ &\times\left(M^4_{\alpha_1\alpha_2\alpha_3\alpha_4}-M^2_{\alpha_1\alpha_2}M^2_{\alpha_3\alpha_4}-M^2_{\alpha_1\alpha_3}M^2_{\alpha_2\alpha_4}-M^2_{\alpha_1\alpha_4}M^2_{\alpha_2\alpha_3}\right)\\
&+A\vec{b}.
\end{split}
\end{equation}
At the isotropic state it must hold that $A\vec{b}=\vec{0}$, so that the fourth derivative at $\vec{b}=\vec{0}$ is given by the remaining terms. 

\section*{Acknowledgement}
The author would like to thank John Ball for insightful discussions relating to the work in this paper. Also the author would like to thank Apala Majumdar and Tim Sluckin for interesting related discussions that took place during the MLC Young Researchers Meeting at the Isaac Newton Institute in Cambridge, as well as the organisers of the MLC programme for providing the forum for the discussions. The research leading to these results has received funding from the European Research Council under the European Union's Seventh Framework Programme (FP7/2007-2013) / ERC grant agreement n$^{\circ}$ 291053.

\bibliography{EntropyOrderParametersStatPhys}

\end{document}